\documentclass{article}% Math and Physical Sciences Reference Style
\usepackage[margin=1in]{geometry}
\usepackage{graphicx}%
\usepackage{multirow}%
\usepackage{amsmath,amssymb,amsfonts}%
\usepackage{amsthm}%
\usepackage{mathrsfs}%
\usepackage[title]{appendix}%
\usepackage{xcolor}%
\usepackage{textcomp}%
\usepackage{manyfoot}%
\usepackage{booktabs}%
\usepackage{algorithm}%
\usepackage{algorithmicx}%
\usepackage{algpseudocode}%
\usepackage{listings}%
\usepackage[colorlinks = true]{hyperref}

\newtheorem{remark}{Remark}%

\newtheorem{definition}{Definition}%

\usepackage[utf8]{inputenc}
\usepackage{thm-restate}
\usepackage{enumerate}
\usepackage{subcaption}
\usepackage{caption}
\usepackage{bbm}
\usepackage{enumitem}
\usepackage{breqn}
\usepackage{cleveref}
\usepackage{pifont}% http://ctan.org/pkg/pifont

\newtheorem*{claim*}{Claim}

\newtheorem{lemma}{Lemma}
\newtheorem*{lemma*}{Lemma}

\newtheorem{theorem}{Theorem}
\newtheorem*{theorem*}{Theorem}
\newtheorem{defn}{Definition}
\newtheorem*{defn*}{Definition}

\newtheorem*{convention*}{Convention}
\newtheorem{corollary}{Corollary}

\newtheorem{prop}{Proposition}

\newtheorem{assumption}{Assumption}
\newtheorem*{assumption*}{Assumption}

\def\amax{\alpha_{\mathrm{max}}}
\def\SC{\mathsf{TOC}}
\def\A{\mathcal{A}}

\def\R{\mathcal{R}}

\def\R{\mathbb{R}}

\def\Z{\mathbb{Z}}

\def\eps{\varepsilon}

\def\E{\mathbb{E}}
\def\P{\mathbb{P}}

\def\sc{\mathsf{oc}}

\def\mytheta{\theta}

\mathcode`l="8000
\begingroup
\makeatletter
\lccode`\~=`\l
\DeclareMathSymbol{\lsb@l}{\mathalpha}{letters}{`l}
\lowercase{\gdef~{\ifnum\the\mathgroup=\m@ne \ell \else \lsb@l \fi}}%
\endgroup

\def\amax{\alpha_{\mathrm{max}}}
\def\Var{\text{Var}}

\newcommand\abs[1]{\left\lvert{#1}\right\rvert}
\newcommand\norm[1]{\left\lVert{#1}\right\rVert}
\newcommand\ceil[1]{\left\lceil{#1}\right\rceil}

\newcommand{\rev}[1]{{\color{black}{#1}}}
\newcommand*\samethanks[1][\value{footnote}]{\footnotemark[#1]}

\author{Billy Jin\thanks{School of Operations Research and Information Engineering, Cornell University, Ithaca, NY, USA. Email: \texttt{\{bzj3,ks2375,mx229\}@cornell.edu}.} \and Katya Scheinberg\samethanks[1] \and Miaolan Xie\samethanks[1]}

\title{
 Sample Complexity Analysis for Adaptive Optimization Algorithms with Stochastic Oracles
}

\begin{document}
\maketitle

\begin{abstract}
    	Several classical adaptive optimization algorithms, such as line search and trust-region methods, have been recently extended to stochastic settings where function values, gradients, and Hessians in some cases, are estimated via stochastic oracles. Unlike the majority of stochastic methods, these methods do not use a pre-specified sequence of step size parameters, but adapt the step size parameter according to the estimated progress of the algorithm and use it to dictate the accuracy required from the stochastic oracles. The requirements on the stochastic oracles are, thus, also adaptive and the oracle costs can vary from iteration to iteration. The step size parameters in these methods can increase and decrease based on the perceived progress, but unlike the deterministic case they are not bounded away from zero due to possible oracle failures, and bounds on the step size parameter have not been previously derived. This creates obstacles in the total complexity analysis of such methods, because the oracle costs are typically decreasing in the step size parameter, and could be arbitrarily large as the step size parameter goes to 0.  Thus, until now only the total iteration complexity of these methods has been analyzed. In this paper, we derive a lower bound on the step size parameter that holds with high probability for a large class of adaptive stochastic methods. We then use this lower bound to derive a framework for analyzing the expected and high probability total oracle complexity of any method in this class. Finally, we apply this framework to analyze the total sample complexity of two particular algorithms, STORM \cite{blanchet2019convergence} and SASS \cite{sass_arxiv}, in the expected risk minimization problem.
\end{abstract}

\section{Introduction}

The widespread use of stochastic optimization algorithms for problems arising in machine learning and signal processing has made the stochastic gradient method and its variants overwhelmingly popular despite their theoretical and practical shortcomings.  \emph{Adaptive stochastic optimization} algorithms, on the other hand,  borrow from decades of advances in deterministic optimization research, and offer new paths forward for stochastic optimization to be more effective and even more applicable. Adaptive algorithms can avoid many of the practical deficiencies of contemporary methods (such as the tremendous costs of tuning the step sizes of an algorithm for each individual application) while possessing strong convergence and worst-case complexity guarantees in surprisingly diverse settings. 

Adaptive optimization algorithms have a long and successful history in deterministic optimization and include line search, trust-region methods, cubic regularized Newton methods, etc. All these methods have a common iterative framework, where at each iteration a candidate step is computed by the algorithm based on a local model of the objective function and a step size parameter that controls the length of this candidate step.   The candidate step is then evaluated in terms of the decrease it achieves in the objective function, with respect to the decrease that it was expected to achieve based on the model. Whenever the decrease is sufficient, the step is accepted and the step size parameter may be increased to allow the next iteration to be more aggressive. If the decrease is not sufficient (or not achieved at all) then the step is rejected and a new step is computed using a smaller step size parameter. The model itself remains unchanged if it is known that a sufficiently small step will always succeed, as is true, for example with first- and second-order Taylor models of smooth functions. The analysis of such methods relies on the key property that the step size parameter is bounded away from zero and that once it is small enough the step is always accepted and the objective function gets reduced. 

When stochastic oracles are used to estimate the objective function and its derivatives, the models no longer have the same property as the Taylor models; steps that decrease the model might not decrease the function, no matter how small the step size is. Thus all stochastic variants of these methods recompute the model at each iteration. The requirement on the stochastic model is then that it achieves a Taylor-like approximation with sufficiently high probability. This also means that steps may get rejected even if the step size parameter is small, simply because the stochastic oracles fail to deliver desired accuracy. Despite this difficulty, one can develop and analyze stochastic variants of adaptive optimization methods.

Recently developed stochastic variants of line search (which we will call step search, since unlike the deterministic version the search direction has to change on each iteration, thus the algorithm is not searching along a line) include \cite{CS17, paquette2018stochastic, berahas2019global, sass_arxiv}.  Stochastic trust-region methods have been analyzed in \cite{blanchet2019convergence, gratton2018complexity, tr_witherrors}, and adaptive cubic regularized methods based on random models has been analyzed in \cite{CS17, sarc}. For all these methods,  bounds on iteration complexity have been derived, either in expectation or in high probability, under the assumption that stochastic oracles involved in estimating the objective function deliver sufficiently high accuracy with sufficiently high probability. 
While this probability is usually fixed, the accuracy requirement of the oracles is adaptive and depends on the step size parameter. Specifically, the smaller that parameter is, the more accurate the oracles need to be, in order to maintain the Taylor-like behavior of the model used by the algorithm. In most applications, having more accurate stochastic oracles implies a higher per-iteration cost. Furthermore, unlike the deterministic case, the step sizes in the stochastic case are not bounded away from zero due to possible oracle failures, and bounds on the step size parameter have not been previously derived. This creates significant difficulty in the analysis of the total oracle complexity of such methods because the oracle costs could be arbitrarily large as the step size parameter goes to zero. 
   
In this paper, we derive a lower bound on the step size parameter for a general class of stochastic adaptive methods that encompasses all the algorithms in the preceding paragraph. This enables us to derive a bound on the total oracle complexity for any algorithm within this class and specific stochastic oracles arising, for example, from expected risk minimization. 
  Our key contributions are as follows:
  \begin{itemize}
  \item Provide a high probability lower bound on the step size parameter for a wide class of stochastic adaptive methods using a coupling argument between the stochastic process generated by the algorithm and a one-sided random walk.
  \item Derive a framework for analyzing expected and high probability total oracle complexity bounds for this general class of stochastic adaptive methods. 
  \item Apply these bounds to STORM \rev{(Stochastic Trust-region Optimization with Random
Models)} \cite{blanchet2019convergence} and SASS \rev{(Stochastic Adaptive Step Search)} \cite{sass_arxiv}  to derive their total sample complexity for expected risk minimization, and show they essentially match the complexity lower bound of first-order algorithms for stochastic non-convex optimization \cite{arjevani2019lower}.
   \end{itemize}
\rev{We note that there is a plethora of  other algorithms in the literature that propose different ways to adaptively vary the step sizes in stochastic optimization.  
Some of these methods, such as \cite{adam, tan2016barzilai}, have step size dynamics that are very different from the framework studied in this paper, hence the analysis of this paper is not needed or applicable.
%Some of these methods have step size dynamics that are quite different from the framework we study in this paper, for example \cite{adam, tan2016barzilai}.
%The step size dynamics of these methods are quite different from the framework we study in this paper. 
%Specifically, the step size parameter is bounded away from zero in these methods and thus the analysis of this paper is not needed or applicable. 
Some other adaptive methods, like \cite{rinaldi2023stochastic,shashaani2018astro,regier2017fast}, have step size dynamics that are more similar. However, since these papers do not provide iteration complexity bounds, we cannot obtain their sample complexity by directly applying our framework. Nonetheless, it would be interesting to explore whether our step size lower bound can be applied to the algorithms in these papers when a certain stopping time is specified. 

%
%It is possible that a modified version of our framework could be able to encompass these algorithms, and we leave this  question as a subject for future research.
%
%However, it is not entirely clear if the step size dynamics of these algorithms fit into our framework. 
%Nonetheless, it is possible that a modified version of our framework could be able to encompass these algorithms, and we leave this interesting question as a subject for future research.
%
% Hdo not exact fit into the framework of our paper, but a modified version may be applicable, and thus can be a subject of future research. step sizes that can go to zero, but they only have asymptotic convergence results (and no iteration complexity bounds), and ...

%However, since these papers do not provide iteration complexity bounds, we cannot obtain their sample complexity by directly applying our framework. It would be interesting to explore whether our step size lower bound can be applied to the algorithms in these papers when a certain stopping time is specified. However, it is not entirely clear if the step size dynamics of these algorithms fit into our framework. For instance, in ASTRO-DF, for large enough iteration $k$, the steps are always successful with probability one. Nonetheless, it is possible that a modified version of our framework could be able to encompass these algorithms, and we leave this interesting question as a subject for future research.
}

We consider a continuous optimization problem of the form
\begin{equation}\label{prob.opt}
  \min_{x \in \R^{m}}\ \phi(x),
\end{equation}
where $\phi$ is possibly non-convex, (twice-)continuously differentiable with Lipschitz continuous derivatives.
Neither function values $\phi(x)$, nor gradients $\nabla \phi(x)$ are assumed to be directly computable. Instead, given any $x \in \R^{m}$, it is assumed that stochastic estimates of $\phi(x)$, $\nabla \phi(x)$, and possibly $\nabla^2 \phi(x)$  can be computed, and these estimates may possess different levels of \emph{accuracy} and \emph{reliability} depending on the particular setting of interest.

The adaptive stochastic algorithms that have been developed recently \cite{CS17, paquette2018stochastic, berahas2019global, blanchet2019convergence, gratton2018complexity, sass_arxiv, tr_witherrors, sqp, qNewton}  use these stochastic estimates to compute an $\varepsilon$-optimal point $x_\varepsilon$, which means $\phi(x_\varepsilon)-\inf_x\phi(x) \leq \varepsilon$ if $\phi$ is convex or  $\norm{\nabla \phi(x_\varepsilon)} \leq \varepsilon$ if $\phi$ is non-convex.
In the next section, we will introduce the general framework that encompasses these methods and discuss particular examples in more detail. In \Cref{sec:stoch}, we discuss the stochastic process generated by the algorithmic framework, including the stochastic step size parameter. In \Cref{sec:stepsize}, we derive a lower bound on the step size parameter which holds in high probability. In \Cref{sec:abstract}, we use this lower bound to derive abstract expected and high probability total oracle complexity for any algorithm in this framework. In \Cref{sec:apply}, we particularize these bounds to the specific examples of first-order STORM \cite{blanchet2019convergence} and SASS \cite{sass_arxiv} algorithms to bound their total sample complexity when applied to expected risk minimization. We conclude with some remarks in \Cref{sec:conclusion}.

\section{Algorithm framework and oracles} \label{sec:algo}

We now introduce and discuss an algorithmic framework for adaptive stochastic optimization in Algorithm \ref{alg:generic_stochastic}. The framework is assumed to have access to stochastic oracles, that for any given point $x$ can generate random quantities $f(x,\xi_0)\approx \phi(x)$ (via zeroth-order oracle), $g(x,\xi_1) \approx \nabla \phi(x)$  (first-order oracle) and, possibly,  $H(x,\xi_2)\approx \nabla^2 \phi(x)$ (second-order oracle). \rev{After we introduce the algorithm we will discuss the  general definition of an oracle that we use in this paper. }

\subsection{General Algorithm}\label{sec:algo1} 

 At each iteration, given $x_k$, a  stochastic local model $m_k(x_k+s) : \R^{m} \to \R{}$ of $\phi(x_k+s)$ is constructed using  $g(x,\xi_1)$ and possibly $H(x,\xi_2)$.   Using this model, a step $s_k(\alpha_k)$  is computed so that $m(x_k+s_k(\alpha_k))$ is a sufficient improvement  over $m_k(x_k)$, where $\alpha_k$ is the step size parameter, which directly or indirectly controls the step length.    The iterate $x_k$ and the trial point $x_k^+ =x_k+s_k(\alpha_k)$ are evaluated using the zeroth-order oracle  values $f(x_k,\xi_{0,k})$ and $f(x_k^+,\xi_{0,k}^+)$.  If these estimates suggest that sufficient improvement is attained, then the step is deemed \emph{successful}, $x_k^+$ is accepted as the next iterate, and the parameter $\alpha_k$ is increased up to a multiplicative factor; otherwise, the step is deemed \emph{unsuccessful}, the iterate does not change, and $\alpha_k$ is decreased by a multiplicative factor. Unlike in the deterministic case, new calls to all oracles are made at each iteration \emph{even when the iterate does not change}.

%We now discuss how various methods fit into the general framework. 
For each method we give the form of $m_k(x_k+s)$, in terms of  
$g_k=g(x_k,\xi_{1,k})$ and  $H_k=H(x_k,\xi_{2,k})$,  $s_k({\alpha_k})$ and the sufficient reduction criterion  in terms of 
 $f_k^0=f(x_k,\xi_{0,k})$ and $f_k^+=f(x_k^+,\xi^+_{0,k})$.

\begin{algorithm}
	\caption{~\textbf{Algorithmic Framework for Adaptive Stochastic Optimization}}
	\label{alg:generic_stochastic}
	
	{\bf 0. Initialization} 
	
	\vspace{0.5mm}
	
	\hspace{0.5cm} Choose $\mytheta \in (0,1)$, $\gamma\in (0, 1)$, $\alpha_{\max} \in (0,\infty)$, $x_0 \in \R^{m}$, and $\alpha_0 \in (0,\alpha_{\max}]$.  Set $k \gets 0$.
	\vspace{0.5mm}
	
	{\bf 1. Determine model and compute step}
	
	\vspace{0.5mm}
	
	\hspace{0.5cm} Construct a stochastic model $m_k$ of $\phi$ at $x_k$ using $f(x_k,\xi_{0,k})$, $g(x_k,\xi_{1,k})$, and (optionally) $H(x_k,\xi_{2,k})$ from \emph{probabilistic zeroth-, first-, and (optionally) second-order oracles}.  Compute $s_k(\alpha_k)$ such that the model reduction $m_k(x_k) - m_k(x_k+s_k(\alpha_k)) \geq 0$ is sufficiently large.
	\vspace{0.5mm}
	
	{\bf 2. Check for sufficient reduction}
	
	\vspace{0.5mm}
	
	\hspace{0.5cm} Set $x_k^+ \gets x_k + s_k(\alpha_k)$ and compute $f(x_k^+,\xi_{0,k}^+)$ as a stochastic estimate of $\phi(x_k^+)$ using a \emph{probabilistic zeroth-order oracle}.  Check if $ f(x_k,\xi_{0,k})- f(x_k^+,\xi_{0,k}^+)$ is sufficiently large (e.g., relative to the model reduction $m_k(x_k)-m_k(x_k^+)$) using a condition parameterized by~$\mytheta$.
	
	\vspace{0.5mm}
	{\bf 3. Successful iteration}
	
	\vspace{0.5mm}
	
	\hspace{0.5cm} If sufficient reduction has been attained (along with other potential requirements), then set $x_{k+1} \gets x_k^+ $ and $\alpha_{k+1} \gets \min\{\gamma^{-1}\alpha_k,\alpha_{\max}\}$.
	\vspace{0.5mm}
	
	{\bf 4. Unsuccessful iteration}
	
	\vspace{0.5mm}
	
	\hspace{0.5cm} Otherwise, set $x_{k+1} \gets x_k$ and $\alpha_{k+1} \gets \gamma \alpha_k$.
	\vspace{0.5mm}
	
	{\bf 5. Next iteration}
	
	\vspace{0.5mm}
	
	\hspace{0.5cm}  Set $k \gets k + 1$ and go to Step~1.
	\vspace{0.5mm}
\end{algorithm}

\rev{
\subsection{General Oracles}\label{sec:oracles} 
Typically in the optimization literature, an oracle is a computational procedure that provides the algorithm with some (estimate of) required information about the objective function. The oracle is endowed with some properties that the algorithm then utilizes. For example, an oracle may be assumed to produce  $\nabla \phi(x)$ exactly -- an assumption that a relevant optimization algorithm (and its analysis) will make use of and without which the algorithm may fail.  Alternatively, an oracle may produce an inexact estimate of $\nabla \phi(x)$, in which case a relevant algorithm will operate  under some specific knowledge about the properties of this estimator; e.g. that it is an unbiased estimator with variance that  is bounded as a function of $\|\nabla \phi(x)\|$, or that it is a deterministic estimator with an error bounded by some known quantity. Algorithms and their analyses differ depending on the properties of the oracles they utilize. The algorithms  analyzed in  \cite{CS17, paquette2018stochastic, berahas2019global, blanchet2019convergence, gratton2018complexity, sass_arxiv, sarc, tr_witherrors, sqp, qNewton} all fit into the framework of Algorithm \ref{alg:generic_stochastic} but under a variety of specific assumptions on the oracles that generate  function, gradient and and Hessian estimates. These assumptions are more complex than is typical in the literature, yet, they are shown to be applicable in many settings. 
For some extensive discussion on these properties, their  comparison and specific examples we refer the reader, for example,  to \cite{tr_witherrors}. Here we
summarize the main ideas.  

As presented above, the algorithmic framework described in Algorithm \ref{alg:generic_stochastic} has access to oracles that generate  random quantities $f(x,\xi_0)\approx \phi(x)$ (zeroth-order oracle), $g(x,\xi_1) \approx \nabla \phi(x)$  (first-order oracle) and, possibly,  $H(x,\xi_2)\approx \nabla^2 \phi(x)$ (second-order oracle). Each of these oracles can have an input, an output and some intrinsic properties. The input is typically the current iterate $x_k$ and the step size parameter $\alpha_k$. These quantities, together with the intrinsic properties of the oracle, define the probability space (and thus the distribution) of the random variable $\xi$.  For example, in the expected risk minimization setting, where the oracles are computed by averaging a minibatch of samples, $\xi_i$ (for $i \in \{0,1,2\}$)  is the random minibatch  used for the $i^{\mathrm{th}}$-order oracle. The probability space and the distribution of $\xi_i$ may depend on $x$ and $\alpha$, since (as we will see in \Cref{sec:apply}) the size of the minibatch may depend on these quantities. 
As another example, the first-order oracles can be computed using (randomized) finite differences \cite{berahas2019theoretical}. In this case, the probability space and the distribution of $\xi_1$ may depend on the current iterate, the step size parameter and the randomness in function value estimates. More examples of stochastic oracles can include robust gradient estimation \cite{Anscombe1960, yin2018byzantine}, SPSA \cite{spsa}, etc. 

In summary, the oracles can be implemented in a variety of ways, depending on the application,  while the algorithmic framework can be agnostic to how exactly the oracles are implemented. The algorithm operates under certain assumptions on the accuracy and reliability of the oracles. The cost of implementing an oracle to satisfy these requirements depends on the application, but it needs to be considered in an overall complexity analysis, which is what we address in this paper. The general oracles in this paper can be described at a high level as follows.
 
 \begin{definition}{\bf Stochastic $j$th-order oracle.} Given an input  $x \in \R^{n}$ and the step size parameter $\alpha$, the oracle computes $\varphi_j(x,\xi_j)$, an estimate of the $j^{\mathrm{th}}$-order derivative $\nabla^j \phi(x)$. In all the algorithms we consider, $\|\varphi_j(x,\xi_j)- \nabla^j\phi(x)\|$ is assumed to be bounded by some quantity (which will be a function of  $\alpha$), with probability at least $1-\delta_j$.
%  where $\delta_j$ is an intrinsic value. 
   Here, $\xi_j$ is a random variable defined on probability space $(\Omega_j, {\cal F}_j, P_j)$ whose distribution depends on the input  $x$ and $\alpha$, and $\delta_j$ is intrinsic to the oracle.
%    and $\P_{\xi_j}$ denotes probability w.r.t.~that distribution. 
    The cost of the oracle is also a  function of $x$ and $\alpha$. 
  \end{definition}

In the next subsection, we describe how various methods fit into the general framework, and what specific requirements they have on the oracles.
} 

\subsection{Step search method}\label{sec:sass}
In the case of the step search (SS) methods in \cite{CS17, berahas2019global,  sass_arxiv} the particulars are as follows. Quantities $f_k^0$  $f_k^+$ and $g_k$ are random outputs of the stochastic oracles and $H_k$ is some positive definite matrix (e.g., the identity). 
\begin{itemize}
\item $m_k(x_k+s)=\phi(x_k)+ g_k^Ts+\frac{1}{2\alpha_k}s^TH_ks$,
\item $s_k(\alpha_k)=-\alpha_k H_k^{-1}g_k$
\item Sufficient reduction:  $f_k^0-f_k^+ \geq -\mytheta g_k^Ts_k(\alpha_k)-r$
\end{itemize}
\rev{Note that although the true function value $\phi(x_k)$ appears in the definition of $m_k$, we do not need to know or query for it to run the algorithm because the function value is just a constant that makes no difference when we minimize the model.} Here $\mytheta \in (0,1)$, and $r$ is a small positive number that compensates for the noise in the function estimates. We will discuss the choice of $r$ after we introduce conditions on the oracle outputs $f(x,\xi_0)$ and $g(x,\xi_{1})$. 
 
 \begin{itemize}
%	\item \textbf{Zeroth-Order Oracle.} $\mathcal{S_0} = \{(A_0, 1 - e^{})\}$ $A_0 = ..$
	\item \textbf{SS.0} \rev{(Step search, zeroth-order oracle)}. Given a point $x$, the oracle computes a (random) function estimate $f(x,\xi_0)$ \rev{(where $\xi_0 = \xi_0(x)$ is the randomness of the oracle, which may depend on the current point $x$)} such that 
	$$
	{\mathbb P_{\xi_0}}\left ( \abs{\phi(x) - f(x, \xi_0)}< \epsilon_f + t\right ) \geq 1-\delta_0(t), 
	$$
	for some $\epsilon_f\geq0$ and any $t>0$. 
	\item \textbf{SS.1} \rev{(Step search, first-order oracle)}. 
	Given a point $x$ and the current \emph{step size parameter} $\alpha$, the oracle computes a (random) gradient estimate $g(x, \xi_1)$ \rev{(where $\xi_1 = \xi_1(x, \alpha)$ is the randomness of the oracle, which may depend on $x$ and $\alpha$)} such that
	$${\mathbb P_{\xi_1}\left (\|g(x, \xi_1)-\nabla \phi(x)\|\leq \max \{\epsilon_g, \min\{\tau, \kappa \alpha\}\|g(x, \xi_1) \| \}\right)\geq 1-\delta_1}$$
	for some nonnegative constants $\epsilon_g$, $\kappa$, $\tau$ and $\delta_1$. 
\end{itemize}

In \cite{CS17}, $\epsilon_f=0$ and $\delta_0(t)\equiv 0$, which means that the zeroth-order oracle is exact, and $r = 0$. In  \cite{berahas2019global},  $\epsilon_f>0$ and $\delta_0(t)\equiv 0$,
which means that the zeroth-order oracle has a bounded  error with probability one, and $r = 2\epsilon_f$.
In \cite{sass_arxiv},  $\epsilon_f>0$ and $\delta_0(t)= e^{-\lambda t}$ for some 
$\lambda >0$. \rev{This means there is no restriction on the error if it is less than $\epsilon_f$}, and the tail of the error decays exponentially beyond $\epsilon_f$. In \cite{sass_arxiv}, $r > 2\,\sup_x {\E_{\xi_0}}\left [ \,\abs{\phi(x) - f(x, \xi_0)}\,\right ]$. \rev{In these works, $\theta$ is a parameter in $(0,1)$ that can be chosen by the user running the algorithm.}

In \cite{CS17}, $\epsilon_g=0$ and $\delta_1<\frac{1}{2}$.  In  \cite{berahas2019global,sass_arxiv}, $\epsilon_g>0$   and  $\delta_1$ is sufficiently small with a more complicated upper bound. 
In  \cite{CS17,berahas2019global}, $\amax$ is finite, thus $\tau$ is $\kappa\amax$. In \cite{sass_arxiv}, $\amax$ is infinity, and $\tau$ is simply assumed to be some constant intrinsic to the oracle.

\subsection{Trust-region method }\label{sec:tr}
Stochastic trust-region (TR) methods that fall into the framework of Algorithm~\ref{alg:generic_stochastic} have been developed and analyzed in \cite{gratton2018complexity, blanchet2019convergence, tr_witherrors}.  
In the case of TR algorithms,  $f_k^0$,  $f_k^+$, $g_k$,  and (possibly) $H_k$ are random outputs of the stochastic oracles, and  
 \begin{itemize}
 \item $m_k(x_k+s)=\phi(x_k)+{g_k^T}s+\frac{1}{2}s^T {H_k}s$, 
\item  $s_k({\alpha_k})=\arg\min_{s:\, \|s\|\leq {\alpha_k}}{m_k(x_k+s)}$
\item 
Sufficient reduction:  $\frac{{f_k^0-f_k^++{r}}}{m_k(x_k)-m_k(x_k+s_k({\alpha_k}))}\geq  \mytheta$
\item Additional requirement for a successful iteration: $\norm{g_k}\geq \theta_2 \alpha_k$, for some $\theta_2>0$. 
\end{itemize}

The requirements for the oracles are as follows. In the case of first-order analysis in \cite{gratton2018complexity, blanchet2019convergence, tr_witherrors}, the following first-order oracle is assumed to be available.
\begin{itemize}
			\item \textbf{TR1.1} \rev{(First-order trust-region, first-order oracle)}. Given a point $x$ and the current \emph{trust-region radius} $\alpha$, the oracle computes a gradient estimate $g(x, \xi_1)$ \rev{(where $\xi_1 = \xi_1(x, \alpha)$ is the randomness of the oracle, which can depend on $x$ and $\alpha$)} such that
		$${\mathbb{P}_{\xi_1}\left (\|g(x, \xi_1)-\nabla \phi(x)\|\leq   \epsilon_g+\kappa_{eg}\alpha\right)\geq 1-\delta_1}.$$
		Here, $\kappa_{eg}$ and $\delta_1$ are nonnegative constants. 
\end{itemize}
In the second-order analysis, the following first- and second-order oracles are used:
\begin{itemize}
			\item \textbf{TR2.1} \rev{(Second-order trust-region, first-order oracle)}. Given a point $x$ and the current \emph{trust-region radius} $\alpha$, the oracle computes a gradient estimate $g(x, \xi_1)$ \rev{(where $\xi_1 = \xi_1(x, \alpha)$ is the randomness of the oracle, which can depend on $x$ and $\alpha$)} such that
		$${\mathbb{P}_{\xi_1}\left (\|g(x, \xi_1)-\nabla \phi(x)\|\leq   \epsilon_g+\kappa_{eg}\alpha^2\right)\geq 1-\delta_1}.$$
		Here, $\kappa_{eg}$ and $\delta_1$ are nonnegative constants. 
					\item \textbf{TR2.2} \rev{(Second-order trust-region, second-order oracle)}. Given a point $x$ and the current \emph{trust-region radius} $\alpha$, the oracle computes a Hessian estimate $H(x, \xi_2)$ \rev{(where $\xi_2 = \xi_2(x, \alpha)$ is the randomness of the oracle, which can depend on $x$ and $\alpha$)} such that
		$${\mathbb{P}_{\xi_2}\left (\|H(x, \xi_2)-\nabla \phi(x)\|\leq \epsilon_h+ \kappa_{eh}\alpha\right)\geq 1-\delta_2}.$$
		Here, $\kappa_{eh}$ and $\delta_2$ are nonnegative constants. 
\end{itemize}
$\epsilon_h$ and $\epsilon_g$ are assumed to equal $0$ in  \cite{gratton2018complexity, blanchet2019convergence} but are allowed to be positive in \cite{tr_witherrors}. 

In terms of the zeroth-order oracles, the three works make different assumptions. Specifically, in \cite{gratton2018complexity}, as in \cite{CS17}, the zeroth-order oracle is assumed to be exact. 
In \cite{tr_witherrors} the zeroth-order oracle is the same as in \cite{sass_arxiv}, and $r > 2\epsilon_f + \frac{2}{\lambda}\log 4$. For the first-order analysis in \cite{blanchet2019convergence}, however, the zeroth-order oracle is as follows (and $r = 0$).
\begin{itemize}
	\item \textbf{TR1.0} \rev{(First-order trust-region, zeroth-order oracle)}. Given a point $x$ and the current \emph{trust-region radius} $\alpha$, the oracle computes a function estimate $f(x, \xi_0)$ \rev{(where $\xi_0 = \xi_0(x, \alpha)$ is the randomness of the oracle, which can depend on $x$ and $\alpha$)} such that
	$$\P_{\xi_0}\left(\abs{f(x, \xi_0) - \phi(x)} \leq \kappa_{ef} \alpha^2\right) \geq 1 - \delta_0,$$
	where $\kappa_{ef}$ and $\delta_0$ are some nonnegative constants.
\end{itemize}
For the second-order analysis in \cite{blanchet2019convergence}, the zeroth-order oracle requirements are tighter. 
\begin{itemize}
	\item \textbf{TR2.0} \rev{(Second-order trust-region, zeroth-order oracle)}. Given a point $x$ and the current \emph{trust-region radius} $\alpha$, the oracle computes a function estimate $f(x, \xi_0)$ \rev{(where $\xi_0 = \xi_0(x, \alpha)$ is the randomness of the oracle, which can depend on $x$ and $\alpha$)}  such that
	$$\P_{\xi_0}\left(\abs{f(x, \xi_0) - \phi(x)} \leq \kappa_{ef1} \alpha^3\right) \geq 1 - \delta_0$$
	and 
	$$\E_{\xi_0}\left[\abs{f(x, \xi_0) - \phi(x)}\right ] \leq \kappa_{ef2} \alpha^3 $$
	where $\kappa_{ef1}$, $\kappa_{ef2}$ and $\delta_0$ are some nonnegative constants.
\end{itemize}
%In \cite{gratton2018complexity,blanchet2019convergence}, $r = 0$. In  \cite{tr_witherrors}, .

\subsection{Cubicly regularized Newton method}\label{sec:sarc}
The cubicly regularized (CR) Newton methods in \cite{CS17, sarc} also fit the framework of Algorithm~\ref{alg:generic_stochastic} with 
\begin{itemize}
\item $m_k(x_k+s)=\phi(x_k)+{g_k^T}s+\frac{1}{2}s^T {H_k}s+\frac{1}{3{\alpha_k}}\|s\|^3$,
\item  $s_k({\alpha_k})=\arg\min_{s}{ m_k(x_k+s)}, \hfill \refstepcounter{equation}(\theequation)\label{sarc_step}$
\item 
Sufficient reduction:  $\frac{{f_k^0-f_k^+ \rev{+r}}}{m_k(x_k)-m_k(x_k+s_k({\alpha_k}))}\geq  \mytheta$.
\end{itemize}
\rev{In \cite{CS17}, the zeroth-order oracle is assumed to be exact, that is $f_k^0=\phi(x_k)$ and $f_k^+=\phi(x_k+s_k({\alpha_k}))$, and $r=0$.
In \cite{sarc} the zeroth-order oracle and the choice of $r$ are the same as in \cite{sass_arxiv}.
In \cite{CS17} and \cite{sarc}, a version of the following first- and second-order oracles are used. }
%\textbf{Cubic Regularization CR.} This algorithm was studied in \cite{CS17}, and uses probabilistic zeroth-, first-, and second-order oracles. The version of the oracles used in \cite{CS17} are as follows:
\rev{For specific implementation details, please refer to \cite{CS17, sarc}. }
\begin{itemize}
	\item \textbf{CR.1} \rev{(Cubicly regularized Newton, first-order oracle)}. Given a point $x$ and the current parameter $\alpha$, the oracle computes a gradient estimate $g(x, \xi_1)$ \rev{(where $\xi_1 = \xi_1(x, \alpha)$ is the randomness of the oracle, which can depend on $x$ and $\alpha$) such that
	$$\mathbb{P}_{\xi_1}\left (\|\nabla \phi(x)-g(x,\xi_1))\|\leq  \kappa_{eg} \max\left\{ \alpha,\|s\|^2\right\}\right) \geq 1 - \delta_1,$$
	
%	$$
%	\mathbb{P}_{\xi_1}\left (\|g(x, \xi_1)-\nabla \phi(x)\|\leq  \kappa_{eg}\alpha^2\right) \geq 1 - \delta_1,
%	$$
	where $\kappa_{eg}$ and $\delta_1$ are nonnegative  constants, and $s$ is defined in \eqref{sarc_step}.}
		\item \textbf{CR.2} \rev{(Cubicly regularized Newton, second-order oracle)}. Given a point $x$, and the current parameter $\alpha$, the oracle computes a Hessian estimate $H(x, \xi_2)$ \rev{(where $\xi_2 = \xi_2(x, \alpha)$ is the randomness of the oracle, which can depend on $x$ and $\alpha$) such that 
	$$\mathbb{P}_{\xi_2}\left (\|(\nabla^2 \phi(x)-{H(x, \xi_2)})s\|\leq  \kappa_{eh}
		\max\left \{\alpha, \|s\|^2\right \}\right) \geq 1 - \delta_2,$$
			
%	$$
%	\mathbb{P}_{\xi_2}\left (\|H(x, \xi_2)-\nabla^2 \phi(x)\|\leq  \kappa_{eh}\alpha\right) \geq 1 - \delta_2,
%	$$
	where $\kappa_{eh}$ and $\delta_2$ are nononegative  constants, and $s$ is defined in \eqref{sarc_step}. }
\end{itemize}
%\begin{remark}
%The actual conditions on the oracles in \cite{CS17} are different from what we present above and are as follows:
%\begin{align}
%\label{eq:sarc_original}
%\mathbb{P}_{\xi_1, \xi_2}\left (\|g(x, \xi_1)-\nabla \phi(x)\|\leq  \kappa_{eg}\norm{s}^2 \text{ and } \|(H(x, \xi_2)-\nabla^2 \phi(x))s\|\leq  \kappa_{eh}\norm{s}^2\right) \geq 1 - \delta, 
%\end{align}
%for some   $\kappa_{eg}$ , $\kappa_{eh}$ and $\delta$. Here, $s$ is the trial step obtained from minimizing the cubicly regularized model. 
%%Condition \eqref{eq:sarc_original} is exactly the requirement for the stochastic cubic regularization algorithm in \cite{CS17}. 
% 
%By taking $\delta_1 = \delta_2 = \frac{\delta}{2}$ and $\alpha = {\cal O}(1/\sigma)$ (where $\sigma$ is the penalty parameter used in the cubic regularization and the constant in the big-O can be chosen according to Lemma 5.1 of \cite{CS17}), these oracles imply \eqref{eq:sarc_original} by Lemma 5.1 of \cite{CS17}.
%%\begin{align}
%%\label{eq:sarc_original}
%%\mathbb{P}_{\xi_1, \xi_2}\left (\|g(x, \xi_1)-\nabla \phi(x)\|\leq  \kappa_{eg}\norm{s}^2 \text{ and } \|(H(x, \xi_2)-\nabla^2 \phi(x))s\|\leq  \kappa_{eh}\norm{s}^2\right) \geq 1 - \delta.
%%\end{align}
%
%
%
%\end{remark}

It is apparent that all of the algorithms that we discussed above rely on oracles whose accuracy requirements change adaptively with $\alpha$. It is also clear that for many settings, \rev{a higher accuracy requirement leads to a higher oracle complexity}. For example, if a stochastic oracle is delivered via sample averaging, then more samples are needed to provide a higher accuracy.  Therefore, to bound the total oracle complexity of the algorithm, we need to bound the accuracy requirement over the iterations, or equivalently provide a lower bound for the parameter $\alpha$.

\subsection{Notions of the stochastic process}
\label{sec:stoch}
When applied to problem \eqref{prob.opt}, Algorithm \ref{alg:generic_stochastic}  generates a stochastic process (with respect to the randomness underlying the stochastic oracles). 
%In particular, let $M_k$ denote the triple $\{\Xi_{0,k}, \Xi_{0,k}^+, \Xi_{1,k}, \Xi_{2,k}\}$, whose realizations are  $\{\xi_{0,k}, \xi_{0,k}^+, \xi_{1,k}, \xi_{2,k}\}$. Algorithm \ref{alg:generic_stochastic} 	generates {a stochastic process 
%$\{(\mathcal{G}_k, \mathcal{S}_k, \mathcal{H}_k, \mathcal{E}_k, \mathcal{E}_k^+, \mathcal{X}_k, \mathcal{A}_k)\}$
Specifically, let $(X_k)_{k\geq 0}$ be the random iterates with realizations $x_k$, let $(G_k)_{k\geq 0}$ be the gradient estimates with realizations $g_k$, and let $(\mathcal{A}_k)_{k \geq 0}$ be the step size parameter values with realizations $\alpha_k$. The prior works that analyze different methods \rev{belonging to the framework of Algorithm~\ref{alg:generic_stochastic}}  define this stochastic process rigorously, with appropriate filtrations. Here for brevity, we will omit those details, as we do not use them in the analysis. 
%	with realizations $(g_k, s_k, H_k,e_k, e_k^+,  x_k, \alpha_k)$ adapted to the filtration $\{{\cal F}_k:\, k\geq 0\}$, where ${\cal F}_k=\sigma (M_0, M_1, \ldots, M_k)$. At iteration $k$, $\mathcal{G}_k$ is the random gradient, $\mathcal{S}_k$ is the random candidate step, $\mathcal{H}_k$ is the random Hessian, $\mathcal{E}_k, \mathcal{E}_k^+$ are the random noises in absolute value of the zeroth-order oracle at $x_k$ and $x_k^+$, $\mathcal{X}_k$ is the random iteration point at step $k$ and $\mathcal{A}_k$ is the random parameter which we will provide a bound for.
%	Note that $\mathcal{G}_k$ is dictated by $\Xi_{1,k}$ in the first-order oracle, and $\mathcal{H}_k$ is dictated by  $\Xi_{2,k}$ in the second-order oracle.
%We define the following random variables, measurable with respect to ${\cal F}_k$.
We now define a stopping time for the process.
\begin{defn}[Stopping time]
	For $\varepsilon > 0$, let $T_\varepsilon$ be the first time such that a specified optimality condition is satisfied. For all the settings considered in this paper,
	${T_\varepsilon}=\min\{k: \norm{\nabla \phi(x_k)} \leq \varepsilon\}$ if $\phi$ is non-convex, and ${T_\varepsilon}=\min\{k: \phi(x_k) - \inf_x \phi(x)  \leq \varepsilon\}$ if $\phi$ is strongly convex. We will refer to $T_\varepsilon$ as the \emph{stopping time} of the algorithm.
\end{defn}
%\begin{defn}[True iteration]
% We also say that iteration $k$ is a {\em true} iteration if all oracles used on that iteration deliver the required accuracy. 
%\end{defn}
 The following property is crucial in the analysis of algorithms in the framework of \Cref{alg:generic_stochastic}.

\begin{assumption}
	[Properties of the stochastic process generated by the adaptive stochastic algorithm]
	\label{ass:alg_behave}
	The random sequence of parameters $\mathcal{A}_k$ generated by the algorithm satisfies the following: 
	\begin{itemize}
		\item[(i)] For all $k$, $\mathcal{A}_k \in \{\gamma \mathcal{A}_{k-1}, \min\{\amax, \gamma^{-1}\mathcal{A}_{k-1}\}\}$, and
		\item[(ii)] There exist constants $\bar\alpha > 0$, and $p>\frac 12$, such that \rev{for all iterations $k$,
%		 for all iterations  $k < T_\varepsilon$, if  $\mathcal{A}_k  \leq  \bar \alpha$ then
		$$\P(\mathcal{A}_{k+1} = \gamma^{-1} \mathcal{A}_{k} \mid \mathcal{F}_{k}, \, k < T_\varepsilon, \, \mathcal{A}_k  \leq  \bar \alpha ) \geq p.$$}
		Here, $\mathcal{F}_k$ denotes the filtration generated by the algorithm up to iteration $k$.
	\end{itemize}       

\end{assumption}
 The algorithms in  \cite{CS17, paquette2018stochastic, berahas2019global, blanchet2019convergence, gratton2018complexity,sass_arxiv,tr_witherrors} all satisfy Assumption \ref{ass:alg_behave}, under appropriate lower bounds on the oracle probabilities  $\delta_0, \delta_1$ (and $\delta_2$). \rev{Also, without loss of generality, we will assume that $\alpha_0 \geq \bar{\alpha}$, since if $\alpha_0 < \bar{\alpha}$ we can simply define $\bar{\alpha}$ to be $\alpha_0$, as $\alpha_0$ is a constant.}
In the next section, under Assumption  \ref{ass:alg_behave}, we derive a high probability lower bound on $\alpha_k$ as a function of the number of iterations $n$, $\bar{\alpha}$, $p$, and $\gamma$. 

Throughout the remainder of this paper, we will use $q$ to denote $1-p$.

\section{High probability lower bound for the step size parameter}
\label{sec:stepsize}

The following theorem provides a high probability lower bound for $\alpha_k$.

\begin{theorem}
\label{thm:stepsize_bound}
	Suppose Assumption \ref{ass:alg_behave} holds for Algorithm \ref{alg:generic_stochastic}. For any positive integer $n$, any $\omega  > 0$, with probability at least  $1 - n^{-\omega } - cn^{-(1+\omega )}$, we have 
	$$\text{either  $T_\eps < n$ ~or~ $\min_{1 \leq k \leq n} \alpha_k  
	\geq \rev{\alpha^*(n)} :=  \bar{\alpha}\gamma  \gamma^{{(1+\omega)\log_{1/2q} n}}
=\bar{\alpha}\gamma   n^{-{(1+\omega)\log_{1/2q} 1/\gamma}}$},$$ 	where $c = \frac{2\sqrt{pq}}{(1-2\sqrt{pq})^2}$ and $q = 1-p$.
	
% 	 In other words, with high probability the algorithm will not reach any step size smaller or equal to $ \gamma^{\ceil{(1+\nu )\log_{p/q} n}}\bar{\alpha}$ in the first $n$ steps. 
\end{theorem}

The proof of this theorem involves two steps. First, in \Cref{sec:randwalk}, we show that for $n<T_\varepsilon$, the sequence of step size parameters $\mathcal{A}_k$ generated by the algorithm can be coupled with a random walk on the non-negative integers. This reduces the problem to that of bounding the maximum value of a one-sided random walk in the first $n$ steps. We then derive  a high probability upper bound on this maximum value in \Cref{sec:max_rand_walk}.

Before moving to its proof, we illustrate the theorem using some plots and comment on some implications of the theorem.

\textbf{Illustration of \Cref{thm:stepsize_bound}.} \Cref{fig:rand_walk} illustrates the high probability bound provided by \Cref{thm:stepsize_bound}. The solid curves depict the lower bounds given by the theorem for $\bar{\alpha} = 1, \omega = 1$, $p = 0.8$, and for varying values of $\gamma$. In comparison, the dotted lines correspond to one-sided random walks
$\mathcal{Z}_k$ that start at $\bar{\alpha} = 1$. 
%If $\mathcal{Z}_k = 1$ then $\mathcal{Z}_{k+1} = \gamma$ deterministically. 
At each step, $\mathcal{Z}_{k+1} = \gamma \mathcal{Z}_k$ with probability $1-p$, and $\mathcal{Z}_{k+1} = \min\{1, \gamma^{-1}\mathcal{Z}_k\}$ with probability $p$. The proof of \Cref{thm:stepsize_bound} shown later implies that there is a coupling between the sequence of parameters $\mathcal{A}_k$ generated by the algorithm and $\mathcal{Z}_k$, such that $\mathcal{A}_k \geq \mathcal{Z}_k$, in other words, the sequence of parameters $\mathcal{A}_k$ generated by the algorithm \emph{stochastically dominates} $\mathcal{Z}_k$. 

% that 1) start at $\bar{\alpha}$, 2) with probability that are meant to capture the \emph{worst-case} (i.e. smallest) that are generated from a one-sided random walk with bias $p$. The main idea of \Cref{thm:stepsize_bound} is to show that these stochastic processes are \emph{stochastically dominated} by the sequence of step-size parameters $\mathcal{A}_k$ generated by the algorithm

\begin{figure}[!ht]
	\begin{subfigure}[t]{0.45\textwidth}
		\vskip 0pt
			\centering
		\includegraphics[scale=0.42]{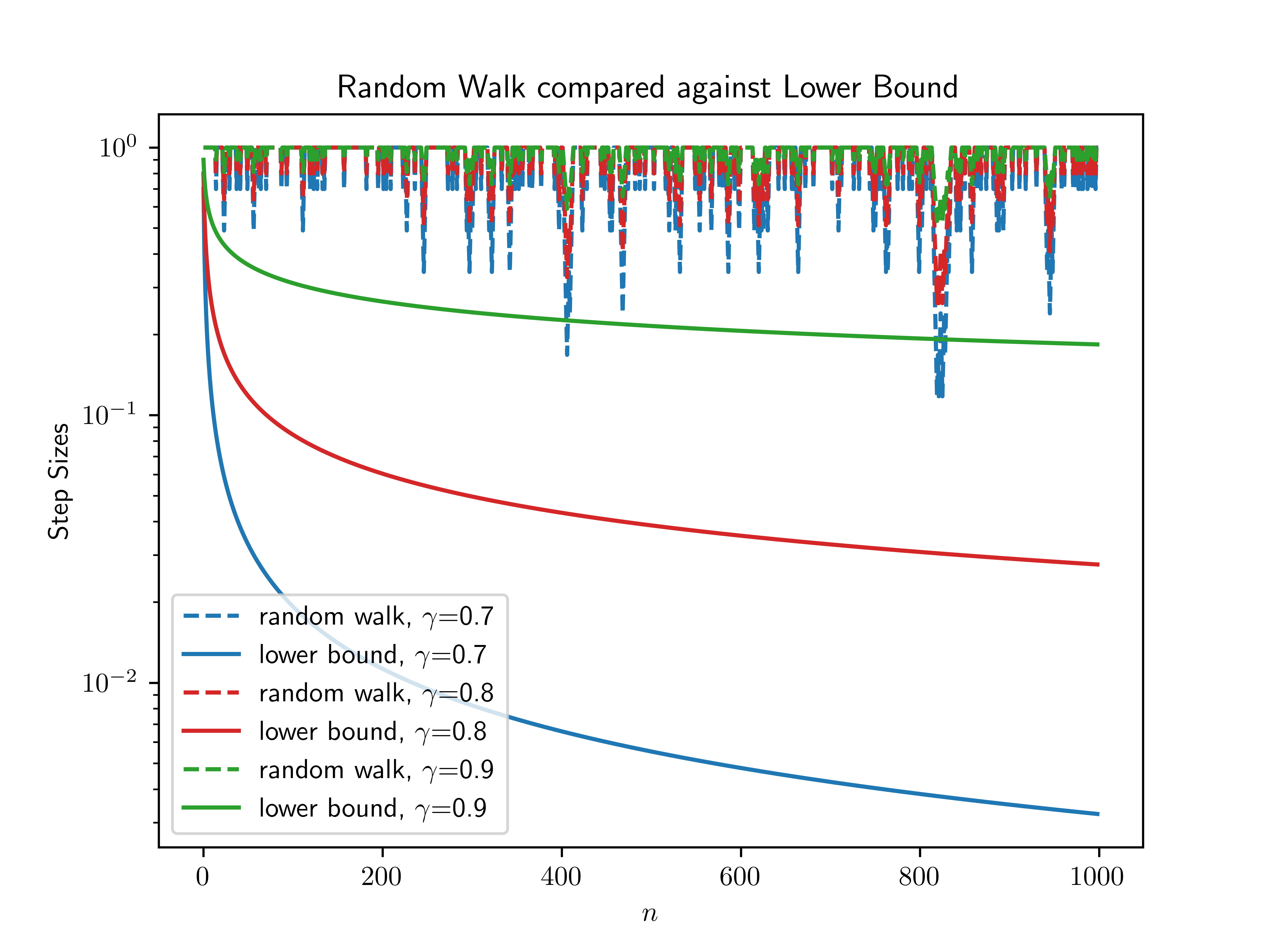}
		\caption{Comparing the random walk trajectory with the theoretical lower bound for various values of $\gamma$.}
		\label{fig:rand_walk}
	\end{subfigure}
    ~
    \begin{subfigure}[t]{0.45\textwidth}
    	\vskip 0pt
			\centering
		\includegraphics[scale=0.42]{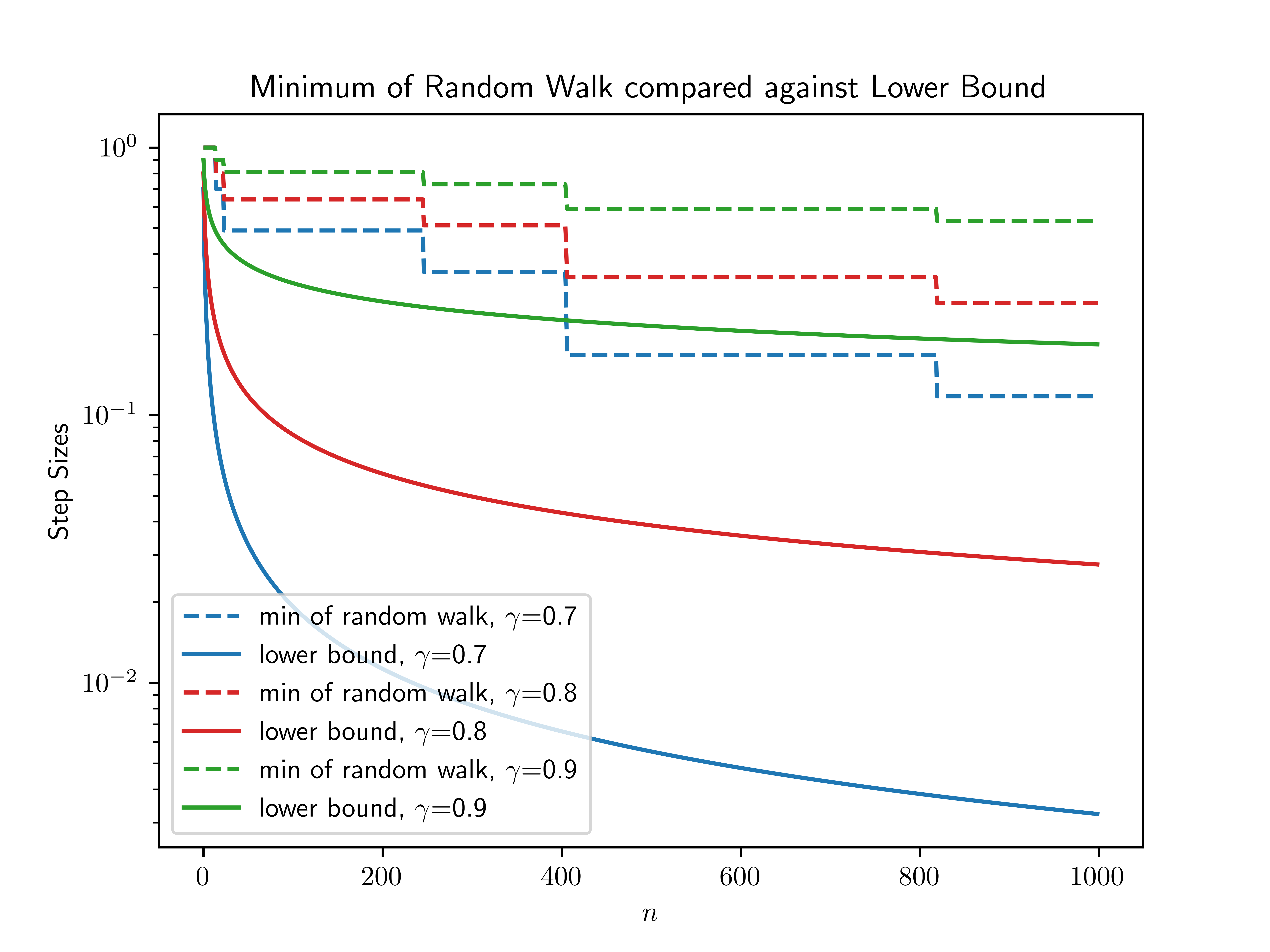}
		\caption{Comparing the \emph{minimum value attained so far by the random walk curves in \Cref{fig:rand_walk}} with the theoretical lower bounds (the same as in \Cref{fig:rand_walk}).}
         \label{fig:min_rand_walk}
    \end{subfigure}
\caption{Illustration of \Cref{thm:stepsize_bound}.}
\end{figure}

\textbf{Remarks on \Cref{thm:stepsize_bound}.} 
\begin{enumerate}
	\item
	For fixed $n, \gamma$, and $\bar\alpha$, the lower bound is a function of $p$. It increases as $p$ increases. Specifically, the exponent of $n$ changes with $p$, and the exponent goes to $0$ as $p$ goes to $1$.
	Hence as $p$ goes to 1, this lower bound simplifies to $\bar{\alpha}\gamma$, which matches the lower bound in the deterministic case.
	
	\item When $p$ is close to $1$ (i.e. when the stochastic oracles are highly reliable), this lower bound decreases slowly as a function of $n$, since the exponent of $n$ is close to $0$. Alternatively, when the stochastic oracles are not highly reliable,  increasing the value of $\gamma$ allows the algorithm to maintain a slow decrease of the step size.
	
	\item Enlarging $\gamma$ as $p$ decreases makes intuitive sense for the algorithm.
	When $p$ is large, an unsuccessful step is more likely to be caused by the step size being too large rather than the failure of the oracles to deliver the desired accuracy. On the other hand, when $p$ is small,
	unsuccessful iterations are likely to occur even when the step size parameter is already small.
	%		 as the oracles get less reliable, their failure is more likely to cause unsuccessful iterations even if the step size parameter is already small.
	Thus in the latter case, larger $\gamma$ values help avoid an erroneous rapid decrease of the step size parameter.
	
	\item  If we  choose 
%	$\gamma=\left(\frac {p} {1-p}\right)^{-\frac 1 4}$ 
	$\gamma=\left(\frac {1} {2q}\right)^{-\frac 1 4}$ \rev{and $\omega = 1$}, 
	then the minimum step size is lower bounded by $\bar{\alpha}\gamma n^{-\frac 1 2}$ with high probability. This coincides with the typical choice of the step size decay schemes for the stochastic gradient method applied to non-convex functions.
\end{enumerate}

The theorem implies that we can even bound $\alpha$ by a \emph{constant} times $\bar{\alpha}$ with high probability, provided we set $\gamma$ as a function of $n$.

\begin{corollary}\label{cor:stepsize_bound}
Let Assumption \ref{ass:alg_behave} hold for Algorithm \ref{alg:generic_stochastic}, then for any positive integer $n$, any $\omega  > 0$, and any $\beta < \frac12$, if 
	$$\gamma \geq \max\left\{\frac12, \left(\frac{1}{2q}\right)^{\frac{\log(2\beta)}{(1+\omega)\log n} } \right\},$$
	then with probability at least  $1 - n^{-\omega } - cn^{-(1+\omega )}$, where $c = \frac{2\sqrt{pq}}{(1-2\sqrt{pq})^2}$,we have 
	$$\text{either $T_\eps < n$ or $\min_{1 \leq k \leq n} \alpha_k  \geq \beta \bar{\alpha}$}.$$
\end{corollary}

\begin{proof}
This follows from \Cref{thm:stepsize_bound} by substituting in the specified value of $\gamma$.
\end{proof}
	 %To get a better sense of this, let’s plug in some typical numbers. 
In the remainder of this section, we prove \Cref{thm:stepsize_bound} in two steps.
  
% For example, suppose $p = \frac23$, $\omega = 1$, $\beta = \frac14$, and $n = \frac{1}{\epsilon^2}$ (which is on the order of the number of iterations required by a typical first-order algorithm to reach an $\epsilon$-stationary point of a smooth function). Then the bound says that as long as $\gamma \geq 2^{-\frac12 \log_{\frac{1}{\epsilon^2}}(2)}$, then with high probability in the first $\frac{1}{\eps^2}$ iterations, either the algorithm succeeds, or the smallest step size will be bounded below by $\frac14 \bar{\alpha}$. If $\epsilon = 10^{-2}$, this says we should choose $\gamma \geq 0.974$. If $\epsilon = 10^{-6}$, this says we should choose $\gamma \geq 0.991$.
%	This tells us that we can specify the high probability lower bound as need by setting the appropriate $\gamma$ parameter in the algorithm.

\subsection{Step 1: reduction to random walk}
\label{sec:randwalk}
	
	We will use a coupling argument to obtain the reduction to a random walk. 

%For the coupling, we partition the iterations $k$ into two parts: $k\leq T_{\eps}$ and $k> T_{\eps}$.
%	\begin{itemize}
%		\item For any $k\leq T_{\eps}$, we use the following coupling.
%		\item  For all $k> T_{\eps}$, we let the step-sizes sequence to emulate exactly the random walk sequence starting at the right depth. This part of the sequence does not have anything to do with the algorithm performance.
%	\end{itemize}
%
%In this way, we obtain an infinite length sequence of step-sizes that couples with infinite length random walks. To be precise

Let $\{\A_k\}_{k=0}^\infty$ denote the random sequence of parameter values (whose realization is $\{\alpha_k\}_{k=0}^\infty$), for  Algorithm \ref{alg:generic_stochastic}.  Let us assume, WLOG, that $\A_0 = \gamma^j \bar{\alpha}$, for some integer $j\leq0$.  (Recall here that $0 <\gamma < 1$.) Then we observe that $\mathcal{A}_k = \gamma^{\bar{Y}_k}\bar{\alpha}$,   where  $\{\bar{Y}_k\}_{k=0}^\infty$ is a random sequence of  integers, with $\bar{Y}_0 =j\leq 0$, which increases by one on every unsuccessful step, and decreases by one on every successful step. Moreover, by 
Assumption  \ref{ass:alg_behave}, whenever $k < T_\eps$ and $\bar{Y}_k \geq 0$, the probability that it decreases by one is at least $p$. Define $Y_k$ as follows:
\begin{equation}
\label{eq:Yk}
Y_k = \bar{Y}_k \quad \text{if $k \leq T_\eps$},
\qquad
Y_k = 
\begin{cases}
Y_{k-1} - 1 &\text{w.p. $p$} \\
Y_{k-1} + 1 &\text{w.p. $1-p$} 
\end{cases}
\quad
\text{if $k > T_\eps$.} 
\end{equation}
In other words, $Y_k$ follows the algorithm until $T_\eps$, and then behaves like a random walk with downward drift $p$ after $T_\eps$.
We now couple  $\{Y_k\}_{k=0}^\infty$  with a random walk $\{Z_k\}_{k=0}^\infty$ which stochastically dominates $Y_k$.

Consider the following one-sided random walk $\{Z_k\}_{k=0}^\infty$, defined  on the non-negative integers. 
\begin{equation}
\label{eq:randwalk_update}
Z_0 = 0, \qquad Z_{k+1} = \begin{cases}
	Z_k +1, &\text{w.p. $1-p$,}\\
	Z_k-1, &\text{w.p. $p$, if $Z_k \geq 1$,}\\
	0, &\text{w.p. $p$, if $Z_k = 0$.}
\end{cases}
\end{equation}

%Here, level 0 corresponds to step size $\bar{\alpha}$, and in general each level $k$ corresponds to step size $\gamma^k\bar{\alpha}$ (recall $0 < \gamma < 1$). Thus, going up in the random walk corresponds to a decrease in the step size. We assume that $p > \frac12$, so that the random walk has a drift towards 0. Let $q := 1-p$.

\begin{lemma}
	\label{lem:coupling_new}
	There exists a coupling between $Z_k$  and $Y_k$, where $Z_k$ stochastically dominates $Y_k$. 
\end{lemma}
\begin{proof}[Proof of Lemma \ref{lem:coupling_new}]
Initially, $Z_0=0$ and $Y_0 \leq 0$. For each $k$, we show how to update $Z_k$ to $Z_{k+1}$ according to how $Y_k$ changes to $Y_{k+1}$. We consider two cases depending on whether $k < T_\eps$ or $k \geq T_\eps$. 

\underline{Case 1: $k < T_\eps$}.
 If $Y_k\leq -1$, we update $Z_{k+1}$ from $Z_{k}$ according to \Cref{eq:randwalk_update}, independently of how $Y_k$ changes to $Y_{k+1}$. If $Y_k \geq 0$, then we first check if $Y_k$ increased or decreased. Let $p'$ be the probability that $Y_{k+1} = Y_k -1$ on this sample path. Since $Y_k \geq 0$, we know by \Cref{ass:alg_behave} that $p' \geq p$. Now, if $Y_{k+1} = Y_k + 1$, then we set $Z_{k+1} = Z_{k} + 1$. On the other hand, if $Y_{k+1} = Y_k - 1$, then we set $Z_{k+1} = Z_k + 1$ with probability $1 - \frac{p}{p'}$, and $Z_{k+1} = \max\{Z_k-1, 0\}$ with probability $\frac{p}{p'}$. Note that these probabilities are well-defined because $p' \geq p$. 
 
\underline{Case 2: $k \geq T_\eps$}. If $Y_{k+1} = Y_k + 1$, then set $Z_{k+1} = Z_{k}+1$. Otherwise, if $Y_{k+1} = Y_{k}-1$, then set $Z_{k+1} = \max\{Z_k-1, 0\}$. 

Observe that under this coupling, $Z_k \geq Y_k$ on every sample path. Moreover, $\{Z_k\}$ and $\{Y_k\}$ have the correct marginal distributions. For $Y_k$, this is easy to see, since it evolves according to its true distribution and we are constructing $Z_k$ from it. For $Z_k$, on any step with $k \geq T_\eps$, $Z_{k+1}$ evolves from $Z_k$ correctly according to \Cref{eq:randwalk_update} by construction. On a step with $k < T_\eps$ there are two cases: 1) $Y_k \leq -1$,  and 2) $Y_k \geq 0$. In the first case, the update from $Z_k$ to $Z_{k+1}$ clearly follows \Cref{eq:randwalk_update}. This is also true in the second case, since there the probability that $Z_{k}$ increases is $(1-p') + p'(1-\frac{p}{p'}) = 1-p$.

To summarize, we have exhibited a coupling between $\{Z_k\}$ and $\{Y_k\}$, under which $Z_k \geq Y_k$ on any sample path. 

%Thus, $\sA_k = \gamma^{Y_k}\bar{\alpha}$ the step size of the algorithm at iteration $k$, stochastically dominates $\gamma^{Z_k}\bar{\alpha}$.  Thus, to obtain a lower bound on the step sizes of the algorithm, it suffices to obtain an upper bound on the random walk $\{Z_k\}$. 
\end{proof}

\subsection{Step 2: upper-bounding the maximum value of the random walk}
\label{sec:max_rand_walk}

We now derive a high probability upper bound on the maximum value reached by the random walk.
\begin{defn}
Let $N(l, n)$ be the random variable that denotes the number of times $Z_k=l$  in the first $n$ steps of the random walk.  
%When $N(l, n)>0$, we say state $l$ is visited by the random walk. 
\end{defn}
By definition of $N(l, n)$, we have $N(l, n) > 0$ if and only if state $l$ is visited in the first $n$ steps of the random walk. The next proposition  upper bounds the probability that $N(l, n) > 0$. 

\begin{prop}\label{prop:Fill}
	Let $q=1-p$. We have
%	$$\P(N(l,n) >0 ) \leq(n-l+1) \frac{1-(q / p)}{1-(q / p)^{l+1}}\left(\frac{q}{p}\right)^{l}  +
%	\frac{2 p \rev{2^l(pq)^{l/2}}}{(1-2 \sqrt{p q})^{2}}\left(\frac{q}{p}\right)^{(l+1) / 2}.$$
	$$\P(N(l,n) >0 ) \leq(n-l+1) \frac{1-(q / p)}{1-(q / p)^{l+1}}\left(\frac{q}{p}\right)^{l}  +
	\frac{2 \sqrt{pq}}{(1-2 \sqrt{p q})^{2}}(2q)^l.$$
\end{prop}

\begin{proof}
%	Let $q=1-p .$ 
%	Denote by $p_{l, N}$ the probability that state $l$ is reached in the first $N$ steps. Note that $p_{l, N} = \P(\mathcal{N}(l, N) > 0)$ by definition. 
	First, observe that $ \P({N}(l, n) > 0)$ remains unchanged if we change the state space from $\{0, 1,2, \ldots\}$ to $\{0, 1,2, \ldots, l\}$ and modify the walk to hold in state $l$ with probability $q$ (instead of moving from $l$ to $l+1$ with that probability). This defines a Markov chain on $\{0,1,2,\ldots, l\}$, and let $P$ be its transition matrix. Noting that $P_{0,l}^m$ is the probability that the Markov chain is in state $l$ at time $m$, we see that
	\begin{equation}
		\label{eq:subadd}
		\P(N(l,n) >0 ) \leq \sum_{m=l}^{n} P_{0, l}^{m}.
	\end{equation}
	
	The matrix $P$ is explicitly diagonalized in \cite{feller_book} (Section XVI.3). By $(3.16)$ in that section,
	\begin{equation}
		\label{eq:transprob}
		P_{0,l}^{m}=\frac{1-(q/p)}{1-(q / p)^{l+1}}\left(\frac{q}{p}\right)^{l} 
		-\frac{2 q}{l+1}\left(\frac{q}{p}\right)^{\frac{l-1}{2}} \sum_{r=1}^{l} \frac{\left[\sin \frac{\pi r}{l+1}\right]\left[\sin \frac{\pi rl}{l+1}\right]\left[2 \sqrt{pq} \cos \frac{\pi r}{l+1}\right]^{m}}{1-2 \sqrt{p q} \cos \frac{\pi r}{l+1}}.
	\end{equation}
	The absolute value of the sum appearing in (\ref{eq:transprob}) can of course be bounded above by
	$$
	\sum_{r=1}^{l} \frac{(2 \sqrt{p q})^{m}}{1-2 \sqrt{p q}}=l \frac{(2 \sqrt{p q})^{m}}{1-2 \sqrt{p q}}
	$$
	and this readily yields
	\begin{equation}
		\label{eq:fillprop_final}
		P_{0, l}^{m} \leq \frac{1-(q/p)}{1-(q / p)^{l+1}}\left(\frac{q}{p}\right)^{l} 
		+
		\frac{2 p}{1-2 \sqrt{p q}}\left(\frac{q}{p}\right)^{\frac{l+1}{2}}(2 \sqrt{p q})^{m}.
	\end{equation}
	Summing (\ref{eq:fillprop_final}) over $m=l, \ldots, n$ and using (\ref{eq:subadd}), we obtain the bound on $\P(N(l,n) >0 )$ claimed in the proposition.
 
\end{proof}

\begin{remark}	The bound for Proposition \ref{prop:Fill} is essentially tight, as the decay of $\P(N(l,n) >0 )$ is not faster than geometric; $q^l$ is a lower bound. 
%Hence, the bound in Theorem \ref{thm:stepsize_bound} is essentially tight.
\end{remark} 

With the above proposition at hand, Theorem \ref{thm:stepsize_bound} is proved by choosing an appropriate level $l$, for which $\P(N(l,n) =0 )$ is high. 
%The full proof of \Cref{thm:stepsize_bound} is as below:
\begin{proof}[Proof of \Cref{thm:stepsize_bound}]
	Let $q=1-p$. By Proposition \ref{prop:Fill}, we have:
	$$\P(N(l,n) >0 ) \leq(n-l+1) \frac{1-(q / p)}{1-(q / p)^{l+1}}\left(\frac{q}{p}\right)^{l}  +
	\frac{2 \sqrt{pq}}{(1-2 \sqrt{p q})^{2}}(2q)^l.$$
	In other words,
	\begin{align*}
		\P(N(l,n) =0 ) &\geq 1-(n-l+1) \frac{1-(q / p)}{1-(q / p)^{l+1}}\left(\frac{q}{p}\right)^{l}  -
		\frac{2 \sqrt{pq}}{(1-2 \sqrt{p q})^{2}}(2q)^l \\
		&\geq 1- n \left(\frac{q}{p}\right)^{l}  -
		\frac{2 \sqrt{pq}}{(1-2 \sqrt{p q})^{2}}(2q)^l.
	\end{align*}
	Let $a>0$ be a parameter to be set later, and take $l=\ceil{a\log(n)}$. Then, the above inequality implies:
	\begin{align*}
	\P(N(l,n) =0 ) &\geq 1- n \left(\frac{q}{p}\right)^{a\log(n)}  -
	\frac{2 \sqrt{pq}}{(1-2 \sqrt{p q})^{2}}(2q)^{a\log(n)} \\
	&= 1- n^{1-a\log(\frac{p}{q})}  -
	\frac{2 \sqrt{pq}}{(1-2 \sqrt{p q})^{2}}n^{-a\log(1/(2q))}.
	\end{align*}
%	Note that $\left(\frac{q}{p}\right)^{\log_a(n)}=n^{\log_a(\frac{q}{p})}$, hence
%	$$\P(N(k,n) =0 )\geq 
%	1- n^{1-\log_a(\frac{p}{q})}   -
%	\frac{2 p\sqrt{q}}{\sqrt{p}(1-2 \sqrt{p q})^{2}}n^{-\frac 1 2\log_a(\frac p q)}.$$
	\rev{In going from the first line to the second, we used the following fact about logarithms: $x^{a \log n} = n^{a \log x} = n^{-a\log (1/x)}$.} Let $a = \frac{1+\omega}{\log (1/2q)}$. Then, $a \log(\frac{p}{q})  \geq  a \log(\frac{1}{2q}) = 1+\omega$, so 
	$$\P(N(l,n) =0 )\geq 
	1- n^{-\omega}   -
	\frac{2 \sqrt{pq}}{(1-2 \sqrt{p q})^{2}}n^{-(1+\omega)}.$$
	In other words, with probability at least $1 - n^{-\omega} - cn^{-(1+\omega)}$ with $c = \frac{2\sqrt{pq}}{(1-2\sqrt{pq})^2}$, the random walk will remain below $l=\ceil{(1+\omega)\log_{1/2q} n}$ in the first $n$ steps. By construction of the coupling, with the above probability, we know the  $\alpha$ parameter in the algorithm either remains above $\gamma^{\ceil{a\log n}}\bar{\alpha} = \gamma^{\ceil{(1+\omega)\log_{1/2q} n}}\bar{\alpha}$  throughout the first $n$ steps, or the algorithm has reached its stopping time in $n$ steps. 
 
\end{proof}

\rev{
	It is natural to ask if the theory extends to the case where the factors for increasing and decreasing the step size are different. The main challenge in this case is that the stochastic process of the step sizes is now modeled by a random walk whose step length going up is different than the step length going down. If it turns out that the hitting probability $\P(N(l,n)>0)$ can be bounded for the more general one-sided random walks, we believe the theory in our paper can be extended similarly. We will leave it as a subject for future research. 
}

\section{Expected and high probability total oracle complexity}
\label{sec:abstract}

We now use the tools derived in the previous section to obtain abstract expected and high probability upper bounds on the total oracle complexity of \Cref{alg:generic_stochastic}. In \Cref{sec:apply}, we will derive concrete bounds for the total oracle complexity of two specific algorithms (STORM and SASS), and the specific oracles arising in expected risk minimization.

%We first prove an upper bound on the expected total oracle complexity, and then we will prove an upper bound that holds with high probability.  
The cost of an oracle call may depend on the step size parameter $\alpha$ and the probability parameter $1-\delta$, thus we denote the cost by $\sc(\alpha, 1-\delta)$. 
We will use $\sc(\alpha)$ in the paper to simplify the notation because for all algorithms in the class, $\delta$ can be treated as a constant.
 Moreover, the cost of an oracle call is a non-increasing function of  $\alpha$ for all algorithms developed so far that fit into the framework. 
 
%\textbf{Key Assumption.}
\begin{assumption}
	\label{ass:monotone}
%	\item For a fixed set of probability parameters of the stochastic oracles, the oracle costs of all the oracle calls at iteration $k$ only depend on the parameter $\alpha_k$. We assume $\sc(\alpha)$ increases as $\alpha$ decreases.
  $\sc(\alpha)$ is non-increasing in $\alpha$.
\end{assumption}
\begin{defn}[Total Oracle Complexity]
For a positive integer $n$, let $\SC(n)$ be the random variable which denotes the total oracle complexity of running the algorithm for $n$ iterations. In other words,
$$\SC(n) = \sum_{k=1}^n \sc(\mathcal{A}_k).$$
\end{defn}

\subsection{Abstract expected total oracle complexity}
We now proceed to bound $\SC(\min\{T_\epsilon, n\})$ in expectation, where $n$ is an arbitrary positive integer.
%First we consider the expected sample complexity of Algorithm \ref{alg:generic_stochastic}. 
%Our main result is as follows:
\begin{theorem}
\label{thm:sc_expected}
Let \Cref{ass:alg_behave} and \Cref{ass:monotone} hold in Algorithm \ref{alg:generic_stochastic}. For any positive integer $n$, we have
%$$\E[\SC(\min\{T_\epsilon, n\}] \leq n\sum_{l=0}^{n-1} \min\left\{1, n\left(\frac{q}{p}\right)^l + \frac{2p}{(1-2\sqrt{pq})^2}\left(\frac{q}{p}\right)^{\frac{l+1}{2}} \right\}\cdot \sc(\bar{\alpha} \gamma^{l+1})
%	+ n \sc(\bar{\alpha}),$$
	$$\E[\SC(\min\{T_\epsilon, n\})] \leq n \,\sum_{l=1}^{n} \min \left\{ 1, n\left(\frac{q}{p}\right)^l + \frac{2 \sqrt{pq}}{(1-2 \sqrt{p q})^{2}}(2q)^l\right \}\cdot \sc(\bar{\alpha} \gamma^{l})
	+ n \,\sc(\bar{\alpha}). $$
	
%where $\sc(\alpha)$ is the total sample cost of evaluating the oracles in any iteration where the step size parameter is equal to $\alpha$. 
\end{theorem}
\begin{proof}
First, observe that if the $\alpha_k$ parameters are all above some value $\alpha^*$ in the first $n$ steps, then by \Cref{ass:monotone}, $\SC(n) \leq n \cdot \sc(\alpha^*)$. 
Therefore, for any integer $l \geq 0$, we have
\begin{align}
\P(\SC(\min\{T_\epsilon, n\}) > n \cdot \sc(\bar{\alpha} \gamma^l)) 
\leq \P(N(l+1,n) > 0).
\end{align}
By Proposition 
\ref{prop:Fill}, 
%\begin{align}
%	\P(N(l,n) > 0) \leq n\left(\frac{q}{p}\right)^l + \frac{2p}{(1-2\sqrt{pq})^2}\left(\frac{q}{p}\right)^{\frac{l+1}{2}}.
%\end{align}
$$\P(N(l+1,n) >0 ) \leq n \left(\frac{q}{p}\right)^{l+1}  +
	\frac{2 \sqrt{pq}}{(1-2 \sqrt{p q})^{2}}(2q)^{l+1}.$$

This implies
%\begin{align}
%	\P(\SC(\min\{T_\epsilon, n\} \geq n \cdot \sc(\bar{\alpha} \gamma^l)) \leq n\left(\frac{q}{p}\right)^l + \frac{2p}{(1-2\sqrt{pq})^2}\left(\frac{q}{p}\right)^{\frac{l+1}{2}}.
%\end{align}
$$ \P(\SC(\min\{T_\epsilon, n\}) > n \cdot \sc(\bar{\alpha} \gamma^l)) \leq \min \left\{ 1, n\left(\frac{q}{p}\right)^{l+1} + \frac{2 \sqrt{pq}}{(1-2 \sqrt{p q})^{2}}(2q)^{l+1}\right \}.$$
%Clearly this probability must also be bounded above by 1, so we have
%\begin{align}
%	\P(\SC(\min\{T_\epsilon, n\} \geq n \cdot \sc(\bar{\alpha} \gamma^l)) \leq \min \left\{ 1, n\left(\frac{q}{p}\right)^l + \frac{2p}{(1-2\sqrt{pq})^2}\left(\frac{q}{p}\right)^{\frac{l+1}{2}} \right \}.
%\end{align}
By the definition of expectation,
\begin{align*}
	& \E[\SC(\min\{T_\epsilon, n\})]\\
	&= \sum_{i=0}^{n \cdot \sc(\bar{\alpha}\gamma^n)}{i \cdot \P(\SC(\min\{T_\epsilon, n\}) = i)} \\
	&\leq \sum_{l=0}^{n-1} \P\left(\SC(\min\{T_\epsilon, n\}) \in (n \cdot \sc(\bar{\alpha} \gamma^l) , n \cdot \sc(\bar{\alpha} \gamma^{l+1})]\right) \cdot n \, \sc(\bar{\alpha} \gamma^{l+1}) \\
	&\qquad+ \P\left(\SC(\min\{T_\epsilon, n\}) \in [0 , n \,\sc(\bar{\alpha})]\right) \cdot n \,\sc(\bar{\alpha}) \\
	&\leq \sum_{l=0}^{n-1} \P\left(\SC(\min\{T_\epsilon, n\})> n \, \sc(\bar{\alpha} \gamma^{l}) \right) \cdot  n \,\sc(\bar{\alpha} \gamma^{l+1}) + 
	n \,\sc(\bar{\alpha})  \\
	&\leq \sum_{l=0}^{n-1} \min \left\{ 1, n\left(\frac{q}{p}\right)^{l+1} + \frac{2 \sqrt{pq}}{(1-2 \sqrt{p q})^{2}}(2q)^{l+1}\right \}\cdot n \, \sc(\bar{\alpha} \gamma^{l+1})
	+ n \,\sc(\bar{\alpha}).
\end{align*}
\end{proof}

\subsection{Abstract high probability total oracle complexity}
\label{abstract:high_prob}
We now proceed to bound $\SC(T_{\eps})$ in high probability, using \Cref{thm:stepsize_bound}.

\begin{theorem}
\label{cor:sc_highprob}
\begin{sloppypar}
	Let \Cref{ass:alg_behave} and \Cref{ass:monotone} hold in Algorithm \ref{alg:generic_stochastic}.
	 For any $\omega  > 0$ and positive integer $n$, with probability at least $1-\P(T_\eps> n)- n^{-\omega } - cn^{-(1+\omega )}$, $$\SC(T_{\eps})\leq n\cdot \sc(\alpha^*(n)),$$ 
	 where $\alpha^*(n)= \bar{\alpha}\gamma   n^{-{(1+\omega)\log_{1/2q} 1/\gamma}}$ , and $c$ is as defined in \Cref{thm:stepsize_bound}.
\end{sloppypar}
	 If $\gamma $ is chosen to be at least $\max\left\{\frac12, \left(\frac{1}{2q}\right)^{\frac{\log(2\beta)}{(1+\omega)\log n} } \right\}$ for some $\beta < \frac12$, then $\alpha^*(n) \geq \beta \bar{\alpha}$, thus \rev{with probability at least $1-\P(T_\eps> n)- n^{-\omega } - cn^{-(1+\omega )}$,}
	 $$\SC(T_{\eps})\leq n\cdot \sc(\beta \bar{\alpha}).$$ 
\end{theorem}
\begin{proof}
%	We now want to derive an upper bound on the sample complexity $\SC(T_{\eps})$ that holds with high probability.  With probability at least $ 1 - \P(T_\eps> N)$, the algorithm reaches the stopping time in $N$ iterations, in other words $T_{\eps}\leq N$ with high probability. 
	
	Let $\SC_{\mathsf{rw}}(n)$ be the total oracle complexity of the first $n$ iterations with the corresponding sequence of parameters $\alpha_k$ induced by the one-sided random walk (that is, the sequence defined by $\alpha_k = \bar{\alpha}\gamma^{Z_k}$, where $Z_k$ is defined in \Cref{sec:randwalk}). In other words, 
	$$\SC_{\mathsf{rw}}(n) = \sum_{k=1}^n \sc(\bar{\alpha}\gamma^{Z_k}).$$
	With probability $1 - \P(T_\eps > n)$, we have $T_\eps \leq n$, which implies
	$$\SC(T_{\eps})\leq \SC_{\mathsf{rw}}(T_\eps) \leq \SC_{\mathsf{rw}}(n).$$
	Here, the first inequality is by \Cref{lem:coupling_new} and \Cref{ass:monotone}, and the second inequality is by $T_\eps \leq n$.
	
	The same arguments used in the proof of \Cref{thm:stepsize_bound} show that with probability at least $1 - n^{-\omega} - c n^{-(1+\omega)}$,  we have $\min_{1 \leq k \leq n} \bar{\alpha}\gamma^{Z_k} \geq \alpha^*(n)$. Thus, with at least this probability, $\SC_{\mathsf{rw}}(n) \leq n\cdot\sc(\alpha^*(n))$. 
	
	Putting these together with a union bound, the result follows.
	
	The second part of the theorem follows from substituting in the specific choice of $\gamma$.
%	Together with Theorem \ref{thm:stepsize_bound}, the result follows.
\end{proof}

\section{Applying to STORM and SASS}\label{sec:apply}
In this section, we demonstrate how the generic oracle complexity bounds in the previous section can be applied to concrete combinations of oracles and algorithms.  We will consider the specific setting of expected risk minimization and two algorithms, first-order STORM and SASS, which are described earlier in the paper and fully analyzed in \cite{blanchet2019convergence} and  \cite{sass_arxiv},  respectively.  For each case, we will state the bounds on $\sc(\alpha)$ as a function of $\alpha$, and use those bounds in conjunction with the known bounds on $T_\eps$ (that have been derived in previous papers), to obtain a bound on the total oracle complexity for each algorithm. 

The results we obtain are the first ones that bound the total oracle complexity of STORM and SASS, and we show that both algorithms are essentially near optimal in terms of total gradient sample complexity. 
When deriving these results, for simplicity of presentation, we omit most of the constants involved in the specific bounds on $T_\eps$ and specific conditions on various algorithmic constants. For all such details, we refer the reader to  \cite{blanchet2019convergence} and  \cite{sass_arxiv}. We will include short comments regarding these constants, but otherwise replace them with a ${\cal O}(\cdot)$ notation. 

 \textbf{Problem Setting: Expected risk minimization (ERM)} can be written as  
$$
\min_{x\in \R^m} \phi(x)=\E_{d \sim \mathcal{D}}[l(x,d)].
$$ 
Here, $x$ represents the vector of model parameters, $d$ is a data sample following distribution $\mathcal{D}$, and $l(x,d)$ is the loss when the model parameterized by $x$ is evaluated on data point $d$. This problem is central in supervised machine learning and other settings such as simulation optimization \cite{kim2015guide}. For this problem, it is common to assume the function $\phi$ is $L$-smooth and is bounded from below, and gradients of functions $\nabla_x l(x, d)$ can be computed for any $d\sim   \mathcal{D}$, so we will consider this setting in this section.

In this setting, the zeroth- and first-order oracles are usually computed as follows: 
\begin{equation}\label{eq:erm_oracle}
f(x, \mathcal{S}) = \frac{1}{|\mathcal{S}|}\sum_{d\in \mathcal{S}}l(x,d),\quad g(x, \mathcal{S}) = \frac{1}{|\mathcal{S}|}\sum_{d\in \mathcal{S}}\nabla_x l(x,d), 
\end{equation}
where $\mathcal{S}$ is the ``minibatch" - that is a set of i.i.d samples from $\mathcal{D}$. Generally, $|\mathcal{S}|$ can be chosen to depend on $x$.

In what follows we will refer to the total number of times an algorithm computes $l(x,d)$ for a specific $x$ and $d$ as its {\em total function value sample complexity} and the number of times the algorithm computes  $\nabla l(x,d)$ as its  {\em total gradient sample complexity}. The total (oracle or sample) complexity of the algorithm is defined as the sum of these two quantities. 

\subsection{Total sample complexity of first-order STORM}
\label{sec:storm}

We first consider the first-order stochastic trust-region method (STORM) as introduced and analyzed 
in \cite{blanchet2019convergence}. 
%In this algorithm, the trust-region radius $\Delta_k$ corresponds to the parameter $\alpha_k$ in our framework.
 The algorithm uses zeroth- and first-order oracles defined in TR1.0 and TR1.1 in \Cref{sec:algo}. Trust-region algorithms are usually applied to nonconvex functions and  the stopping time of STORM is defined as $T_\eps = \min\{k: \norm{\nabla \phi(x_k)} \leq \eps \}$. In Section 3.3 of \cite{blanchet2019convergence}, it is shown  that Assumption \ref{ass:alg_behave} is satisfied with $\bar\alpha= \frac{\eps}{\zeta}$, where $\zeta$ is
 a moderate constant that depends on $\kappa_{eg}$, $L$  and some constant chosen by the algorithm. 
 
In  \cite{blanchet2019convergence}, the oracle costs of STORM in the ERM setting are briefly discussed under the following \textbf{assumptions} on $ l(x, d)$.

\begin{itemize}
\item Function value: It is assumed that there is some $\sigma_f \geq 0$ such that for all $x$, $\Var_{d\sim \mathcal{D}}\left[l(x,d)\right] \leq \sigma_f^2$. 
%Denote $\epsef :=\sup_x {\E_{\xi_0}}\left [ \,\abs{\phi(x) - f(x, \xi_0)}\,\right ]$.
\item Gradient: It is assumed that $\E_{d \sim \mathcal{D}} [\nabla_x l(x, d)] = \nabla \phi(x)$, and that there is some $\sigma_g \geq 0$ such that for all $x$, 
%$
%	\E_{d\sim \mathcal{D}}\norm{\nabla l(x,d)-\nabla\phi(x)}^2\leq M_c+M_v\norm{\nabla\phi(x)}^2,
%$
\begin{equation}\label{BCN_0}
%	\E_{\Xi_1}\norm{g(x, \Xi_1)-\nabla\phi(x)}^2\leq M_c+M_v\norm{\nabla\phi(x)}^2.
	\E_{d\sim \mathcal{D}}\norm{\nabla_x l(x,d)-\nabla\phi(x)}^2\leq \sigma_g^2.
\end{equation} 
%%This assumption is fairly general and is studied in the literature \cite{bottou2018optimization}. 
\end{itemize}
 The cost of each oracle call is the number of samples in the associated minibatch $\mathcal{S}$.
By applying   Chebyshev's inequality it is easy to bound the oracle costs of TR1.0 and TR1.1. 
%Moreover, in Section 3.5 of that paper, the authors show that the oracle costs at each iteration can be bounded as follows:
\begin{itemize}
%	\item Cost of zeroth-order oracle with parameter $\alpha_k$: $\tilde{{O}}\left( \frac{\sigma_f^2}{\epsilon_F^2\alpha_k^4} \log\left(\frac{1}{1 - \sqrt{1 - \delta_0}}\right)\right)$,
%	\item Cost of first-order oracle with parameter $\alpha_k$: $\tilde{{O}}\left(\frac{\sigma_{g}^{2}}{\kappa_{e g}^{2} \alpha_{k}^{2}} \log \left(\frac{1}{1-\sqrt{1 - \delta_1}}\right)\right)$.
\item Cost of TR1.0 with parameter $\alpha$: $\sc_0(\alpha) = \frac{\sigma_f^2}{\delta_0\kappa_{ef}^2\alpha^4}$,
	\item Cost of TR1.1 with parameter $\alpha$: $\sc_1(\alpha) = \frac{\sigma_{g}^{2}}{\delta_1\kappa_{e g}^{2} \alpha^{2}} $.
\end{itemize}
%In this setting, the oracles are obtained by averaging a batch of independent function/gradient estimates, and $\sigma_f^2$ and $\sigma_g^2$ are the variances of an individual function and gradient estimate, respectively. 
%Here, $\kappa_{eg}$ can be any constant, and $\kappa_{ef}$ should be small enough (as dictated by Assumption 3 of \cite{blanchet2019convergence}).
Below we substitute the specific oracle costs into \Cref{thm:sc_expected} to obtain the expected total sample complexity for the first-order STORM algorithm. Specifically, we will bound the total sample complexity of STORM $\E[\SC(\min\{T_\eps, n\})]$  by deriving bounds on the expectation of the total
 function value sample complexity $\SC_0$ and the total gradient sample complexity $\SC_1$, where 
$$\SC_0(n) = \sum_{k=1}^n \sc_0(\mathcal{A}_k), ~\SC_1(n) = \sum_{k=1}^n \sc_1(\mathcal{A}_k)  \text{ ~and~ } \SC(n)=\SC_0(n)+\SC_1(n). $$

%Hence, by plugging in the oracle complexity of $\sc(\Delta_{k})=\tilde{\mathcal{O}}\left(\frac{\sigma_{g}^{2}}{\kappa_{e g}^{2} \Delta_{k}^{2}} \log \left(\frac{1}{1-\sqrt{p}}\right)\right)$, and taking $\gamma\rightarrow 1$, the expectation sample complexity goes to $O(\frac{1}{\eps^2}\log(\frac{1}{\eps}))$.  

\begin{theorem}[Expected Total Sample Complexity Bound of First-Order STORM] 
\label{thm:storm}
Let $p = 1 - \delta_0 - \delta_1$ and $q = 1 - p$. For the first-order STORM algorithm, for any iteration \rev{number} $n \in \Z^+$, and $\gamma > (2q)^{\frac14}$, we have
\begin{align*}
\E[\SC(\min\{T_\eps, n\})] \leq  \; &
{\cal O}\left(n\log_{p/q}(n) \,\left(
\frac{\sigma_f^2}{\eps^4}\,
n^{4\log_{q/p}\gamma}  + \frac{\sigma_g^2 }{\eps^2}\,n^{2\log_{q/p}\gamma}  \right)\right).
\end{align*}

%If $\gamma \geq \max\left\{\frac12, \left(\frac{p}{q}\right)^{\frac{\log(2\beta)}{(1+\omega)\log n} } \right\}$, 
If $\gamma \geq \left(\frac{q}{p}\right)^{\frac{\log c}{\log n}}$ for some constant $c > 1$ (so that $n^{\log_{q/p}\gamma} \leq c$), the above simplifies to be 
\begin{equation}
\label{eq:storm_toc_exp_beta}
\E[\SC(\min\{T_\eps, n\})] \leq n\log(n) \cdot {\cal O}\left(\frac{\sigma_f^2}{\eps^4} + \frac{\sigma_g^2}{ \eps^2} \right).
\end{equation}
\end{theorem}
\begin{proof}
%We bound the oracle costs required by the zeroth- and first-order oracles separately.
%\textbf{Zeroth-order oracle.}

By Theorem \ref{thm:sc_expected}, the total expected cost of the zeroth-order oracle over $n$ iterations is bounded above by:
\begin{align*}
\E[\SC_0(\min\{T_\eps, n\})] &\leq n \,\sum_{l=1}^{n} \min \left\{ 1, n\left(\frac{q}{p}\right)^l + \frac{2 \sqrt{pq}}{(1-2 \sqrt{p q})^{2}}(2q)^l\right \}\cdot \sc_0(\bar{\alpha} \gamma^{l})
	+ n \,\sc_0(\bar{\alpha}) \\
	&\leq n\underbrace{\sum_{l=1}^{n} \min\left\{1, n\left(\frac{q}{p}\right)^l\right\}\cdot \sc_0(\bar{\alpha} \gamma^{l})}_{=:A} \\
	&\qquad\qquad + \frac{2n\sqrt{pq}}{(1-2\sqrt{pq})^2} \underbrace{\sum_{l=1}^{n}\left(2q\right)^{l} \cdot \sc_0(\bar{\alpha} \gamma^{l})}_{=: B}
	+ n\, \sc_0(\bar{\alpha}).
\end{align*}
 
For the zeroth-order oracle, $\sc_0(\alpha) = \frac{\sigma_f^2}{\delta_0\kappa_{ef}^2\alpha^4} = {\cal O}(\frac{\sigma_f^2}{\alpha^4})$. We use  this  to calculate upper bounds for $A$ and $B$. First, we consider $A$. Note that $\min\{1, n(\frac{q}{p})^l\} = 1$, if and only if $l \leq \log_{p/q}(n)$. Therefore, 
\begin{align*}
A &\leq 
\sum_{l=1}^{\log_{\frac{p}{q}}(n)} \sc_0(\bar{\alpha} \gamma^{l}) 
+ 
n\sum_{l \geq \log_{\frac{p}{q}}(n)}\left( \frac{q}{p}\right)^l \sc_0(\bar{\alpha} \gamma^{l}) \\
&\leq \sum_{l=1}^{\log_{\frac{p}{q}}(n)} {\cal O}\left(\frac{\sigma_f^2}{\bar{\alpha}^4\gamma^{4l}}  \right) 
+ 
n\sum_{l \geq \log_{\frac{p}{q}}(n)} \left( \frac{q}{p}\right)^l {\cal O}\left(\frac{\sigma_f^2}{\bar{\alpha}^4\gamma^{4l}} \right) \\
&\leq {\cal O}\left(\frac{\sigma_f^2}{\bar{\alpha}^4}  \right) 
\left(
\sum_{l=1}^{\log_{\frac{p}{q}}(n)} \frac{1}{\gamma^{4l}}
+
n\sum_{l \geq \log_{\frac{p}{q}}(n)} \left( \frac{q}{p}\right)^l \frac{1}{\gamma^{4l}}
\right) \\
&\leq  {\cal O}\left(\frac{\sigma_f^2}{\eps^4}  \right) 
\left(
\log_{\frac{p}{q}}(n)
\cdot \left(\frac{1}{\gamma^4}\right)^{\log_{\frac{p}{q}}(n)} 
+
n\left(\frac{q}{p\gamma^4}\right)^{\log_{\frac{p}{q}}(n)}
\frac{1}{1 - \frac{q}{p\gamma^4}}
\right) \\
&=   {\cal O}\left(\frac{\sigma_f^2}{\eps^4} \, \log_{\frac{p}{q}}(n) 
\, n^{\log_{\frac{p}{q}}\left(\frac{1}{\gamma^4}\right)}  \right) \\
&=   {\cal O}\left(\frac{\sigma_f^2}{\eps^4} \, \log_{\frac{p}{q}}(n) 
\, n^{4\log_{\frac{q}{p}}\gamma}  \right).
\end{align*}
Next we bound $B$. We have 
\begin{align*}
B = \sum_{l=1}^{n}\left(2q\right)^{l} \cdot \sc_0(\bar{\alpha} \gamma^{l}) 
\leq \sum_{l=0}^{\infty}\left(2q\right)^{l} \, {\cal O}\left(\frac{\sigma_f^2}{\bar{\alpha}^4\gamma^{4l}}  \right)
\leq {\cal O}\left(\frac{\sigma_f^2}{\eps^4}  \right) 
\left( \frac{1}{1 - 2q \cdot\frac{1}{\gamma^4}} \right) = {\cal O}\left(\frac{\sigma_f^2}{\eps^4}  \right) .
\end{align*}
Using these bounds on $A$ and $B$ in the  expression for $\E[\SC_0(\min\{T_\epsilon, n\})]$, we obtain the bound on the  
total function value sample complexity as $ {\cal O}\left(\frac{\sigma_f^2}{\eps^4} \, n \log_{\frac{p}{q}}(n) 
\, n^{4\log_{\frac{q}{p}}\gamma}  \right)$.

%\textbf{First-order oracle.} 
A similar calculation using the cost of the first-order oracle yields the bound $ {\cal O}\left(\frac{\sigma_g^2}{\eps^2} \, n \log_{\frac{p}{q}}(n) 
\, n^{2\log_{\frac{q}{p}}\gamma}  \right)$ for $\E[\SC_1(\min\{T_\epsilon, n\})]$. Since $\SC(\min\{T_\epsilon, n\})= \SC_0(\min\{T_\epsilon, n\})+\SC_1(\min\{T_\epsilon, n\})$ by definition, the result follows.

\end{proof}

Let us discuss the implications of Theorem \ref{thm:storm}.  In  \cite{arjevani2019lower}, a lower bound on the total gradient sample complexity for stochastic optimization of non-convex, smooth functions is derived and shown to be, in the worst case,  $\frac{C}{\eps^4}$, for some positive constant $C$.  This complexity lower bound holds even when exact function values $\phi(x)$ are also available. We note that the definition of complexity in \cite{arjevani2019lower} is the smallest number of sample gradient evaluations required to return a point $x$ with $\E[\norm{\nabla \phi (x)}] \leq \eps$, which is different from $\SC(T_{\epsilon})$ which we are aiming to bound here. We believe that the lower bound in \cite{arjevani2019lower} applies to our definition as well, but this is a subject of a separate study.

In \cite{blanchet2019convergence},  it is shown that $\E[T_\eps] \leq \frac{C_1}{\eps^2}$ for some $C_1$ sufficiently large that depends on $\delta_1$, $\delta_0$, $\kappa_{eg}, L$ and some algorithmic constants. 
Thus, if $n = \frac{C_1}{\eps^2}$ in inequality
 \eqref{eq:storm_toc_exp_beta} of Theorem \ref{thm:storm} , as long as $\gamma$ is sufficiently large,  we obtain %$${\cal O}\left( \frac{\sigma_f^2}{\epsilon^6} \cdot \frac{1}{\epsilon^{8\log_{q/p} (\gamma)}}  + \frac{\sigma_g^2}{\epsilon^4} \cdot \frac{1}{\epsilon^{4 \log_{q/p}(\gamma)}}\right)$$
$\E[\SC(\min\{T_\eps, n\})] \leq {\cal O}\left( (\frac{\sigma_f^2}{\eps^6}  + \frac{\sigma_g^2}{\eps^4}) \log(\frac{1}{\eps}) \right).$
In particular, the total gradient sample complexity is ${\cal O}\left(\frac{\sigma_g^2}{\eps^{4}} \log\left(\frac{1}{\eps}\right)\right)$, which essentially matches the complexity lower bound as described in \cite{arjevani2019lower} up to a logarithmic factor. The total function value sample complexity is worse than that of the gradient  if $\sigma_f^2$ is large. However, if $\sigma_f^2\leq {\cal O}(\sigma_g^2\eps^2)$ \rev{(which often happens in practice since $\sigma_g$ usually scales with the dimension of the problem, and tends to be much larger than $\sigma_f$)}, the total sample complexity bound  of
STORM matches the lower bound up to a logarithmic factor.

We note now that choosing   $n = \frac{C_1}{\eps^2}$ in Theorem \ref{thm:storm} does not in fact guarantee that $T_\eps<n$, since
for STORM, only a bound on  $\E[T_\eps]$ has been derived. However, 
this statement can be made true in probability, thanks to  \Cref{cor:sc_highprob}, by simply applying Markov inequality for $n = C_2\, \frac{C_1}{\eps^2}$ (where $C_2 > 1$).  

\begin{theorem}[High Probability Total Sample Complexity Bound of First-Order STORM] 
\label{th:storm_hp} For the first-order STORM algorithm applied to expected risk minimization, let 
$n$ be chosen such that $n \geq C_2\,\frac{C_1}{\eps^2} $ (for some $C_1$ sufficiently large so that $\frac{C_1}{\eps^2} \geq \E[T_\eps]$, and any $C_2>1$), and $\gamma$ be chosen so that $\gamma\geq \max\left\{\frac12, \left(\frac{1}{2q}\right)^{\frac{\log(2\beta)}{(1+\omega)\log n} } \right\}$ (for some $\beta\leq \frac12$, and any $\omega  > 0$). Then, with probability at least	$1-\frac{1}{C_2}- {\cal O}(n^{-\omega })$, 
	
%	\begin{equation}
%	\label{eq:storm_toc_hp}
%	\SC(T_{\eps}) 
%%	 \leq n \sc(\alpha^*(n)) = n{\cal O}\left( \frac{\sigma_f^2}{\alpha^4} + \frac{\sigma_g^2}{\alpha^2}\right) 
%%	 = n{\cal O}\left( \frac{\sigma_f^2}{{\eps^4  \gamma^{4{(1+\omega)\log_{1/2q} n}}
%%}} + \frac{\sigma_g^2}{{\eps^2  \gamma^{2{(1+\omega)\log_{1/2q} n}}
%%}}\right) 
%	= {\cal O}\left(\frac{1}{\eps^2}\left( \frac{\sigma_f^2}{{\eps^4 \cdot \eps^{8{(1+\omega)\log_{2q} \gamma}}
%}} + \frac{\sigma_g^2}{{\eps^2  \cdot \eps^{4{(1+\omega)\log_{2q} \gamma}}
%}}\right) \right)
%	\end{equation}

	\begin{equation}
	\label{eq:storm_toc_hp_beta}
	\SC(T_{\eps})  \leq {\cal O}  \left(\frac{\sigma_f^2}{\eps^6} +  \frac{\sigma_g^2}{\eps^4}\right).
	\end{equation}
\end{theorem}
\begin{proof}
The theorem is a simple application of \Cref{cor:sc_highprob} to the specific setting. 
%$\sc(\alpha) = \frac{M_c}{\delta\epsilon_g^2} + \frac{M_v(1+\kappa\amax)^2} {\delta\kappa^2\alpha^2}$.
\end{proof}

\begin{remark}
\begin{enumerate}
	\item Compared to the expected total sample complexity bound, this high probability bound is smaller by a log factor.
%	\item By using  the bound  $\E[T_\eps] \leq {\cal O}(\frac{1}{\eps^2})$   from \cite{blanchet2019convergence} the above result shows that 
%	the total sample complexity of STORM algorithm achieving $\|\phi(x)\|\leq \eps$ when applied to expected loss minimization is bounded by
%	${\cal O}\left(\frac{\sigma_f^2}{\eps^6}  + \frac{\sigma_g^2}{\eps^4}\right)$ with high probability. 
	\item In \cite{gratton2018complexity},  a first-order  trust-region algorithm similar to STORM with the same first-order oracle  (i.e. TR1.1 with $\epsilon_g = 0$), but with an exact zeroth-order oracle (i.e. TR1.0 with $\kappa_{ef}=\delta_0=0$) is analyzed. In this case, it is shown that $\P(T_\eps > n) \leq \exp(-C_1n)$ (for some constant $C_1$ depends on $\delta_1$), for any $n \geq \frac{C_2}{\eps^2}$ (with some sufficiently large $C_2$). Using a similar application of  \Cref{cor:sc_highprob}, we can show that as long as $\gamma$ is sufficiently large, the total gradient sample complexity of that trust-region algorithm is bounded above by ${\cal O}(\frac{\sigma_g^2}{\eps^4})$ with probability at least $1 - \exp(-C_1n) - {\cal O}(n^{-\omega})$ (which is a significant improvement over the probability in Theorem \ref{th:storm_hp}).

	\item Another first-order trust-region algorithm, with weaker oracle assumptions than those in \cite{gratton2018complexity}  is introduced and analyzed in \cite{tr_witherrors}. This algorithm relies on the first-order oracle  as described in TR1.1 and the zeroth-order oracle  as described in SS.0. For this algorithm, it is shown that  $\P(T_\eps > n) \leq 2\exp(-C_1n) + \exp(-C_2)$ ($C_2$ being any positive constant), where $n =  C_3\frac{C_2}{\eps^2}$ for some sufficiently large $C_3$ and some positive $C_1$ that depends on $\delta_0$ and $\delta_1$. Thus, again, using  \Cref{cor:sc_highprob}  we can show that as long as  $\gamma$ is sufficiently large, the total sample complexity of the first-order trust-region algorithm in \cite{tr_witherrors} is bounded above by ${\cal O}(\frac{\sigma_f^2}{\eps^6} + \frac{\sigma_g^2}{\eps^4})$ with probability at least $1-2 \exp(-C_1n) - \exp(-C_2) - {\cal O}(n^{-\omega})$.

\end{enumerate}

\end{remark}

\subsection{Total sample complexity of SASS}

We now consider the SASS algorithm\footnote{This algorithm was also referred to as ALOE in \cite{aloe_neurips}. Its name has been changed to SASS since \cite{sass_arxiv}.}, analyzed in \cite{berahas2019global,sass_arxiv} and described in \Cref{sec:sass}. 
%The algorithm uses a zeroth-order oracle and a first-order oracle, whose requirements are described in SS.0 and SS.1 in \Cref{sec:sass}.  
By Proposition 1, 2 and 4 of \cite{sass_arxiv},  \Cref{ass:alg_behave} is satisfied, with $\bar{\alpha}$ as given in the propositions. 

%In \cite{sass_arxiv}  an upper bound on $T_\eps$ that hold with high probability is  derived.  
%Moreover, in \cite{sass_arxiv}, it is shown that the oracles required by the SASS algorithm can be implemented in the expected loss minimization setting by using a minibatch of an appropriate size. Using these results together with the theorems in this paper, we derive expected and high probability sample complexity bounds for SASS applied to expected loss minimization. 

%\textbf{Problem Setting.}
%Expected loss minimization  can be written as  
%$$
%\min_{x\in \R^n} \phi(x)=\E_{d \sim \mathcal{D}}[l(x,d)].
%$$ 
%Here, $x$ represents the vector of model parameters, $d$ is a data sample following distribution $\mathcal{D}$, and $l(x,d)$ is the loss when the model parameterized by $x$ is evaluated on data point $d$. This problem is the central problem in supervised machine learning and other settings, such as simulation optimization \cite{}. 

%Function $\nabla \phi$ is $L$-smooth and is bounded from below. 
%Gradients  of functions $\nabla_x l(x, d)$ can be computed for any $d\sim   \mathcal{D}$. 

In the empirical risk minimization  setting,  the following  \textbf{assumptions} on $ l(x, d)$ are made in \cite{sass_arxiv}.

\begin{itemize}
\item Function value: It is assumed that $\abs{l(x,d) - \phi(x)}$ is a subexponential random variable and that there is some $\sigma_f \geq 0$ such that for all $x$, $\Var_{d\sim \mathcal{D}}\left[l(x,d)\right] \leq \sigma_f^2$. \\
For example, if $l(x,d)$ is uniformly bounded, then $\abs{l(x,d) - \phi(x)}$ is subexponential.
%Denote $\epsef :=\sup_x {\E_{\xi_0}}\left [ \,\abs{\phi(x) - f(x, \xi_0)}\,\right ]$.
\item Gradient: It is assumed that $\E_{d \sim \mathcal{D}} [\nabla_x l(x, d)] = \nabla \phi(x)$, and that for some $M_c, M_v\geq 0$ and for all $x$, 
%$
%	\E_{d\sim \mathcal{D}}\norm{\nabla l(x,d)-\nabla\phi(x)}^2\leq M_c+M_v\norm{\nabla\phi(x)}^2,
%$
\begin{equation}\label{BCN}
%	\E_{\Xi_1}\norm{g(x, \Xi_1)-\nabla\phi(x)}^2\leq M_c+M_v\norm{\nabla\phi(x)}^2.
	\E_{d\sim \mathcal{D}}\norm{\nabla_x l(x,d)-\nabla\phi(x)}^2\leq M_c+M_v\norm{\nabla\phi(x)}^2.
\end{equation} 
This assumption is fairly general and is studied in the literature \cite{bottou2018optimization}. 
\end{itemize}

For non-convex functions, the stopping time is defined as $T_\eps = \min\{k: \norm{\nabla \phi(x_k)} \leq \eps \}$, same as in the case of STORM. 
For strongly convex functions, the stopping time is defined as $T_{\eps} = \min\{k: \phi(x_k) - \inf_x \phi(x) \leq \eps\}$. To achieve the desired accuracy,  oracles SS.0 and SS.1 have to be sufficiently accurate in the sense that $\epsilon_f$ and $\epsilon_g$ have to be sufficiently small with respect to $\eps$. In the case of expected risk minimization, 
the oracles can be implemented for any $\epsilon_f$ and $\epsilon_g$ by choosing an appropriate mini-batch size. Thus, here we will first fix $\eps$ and then 
discuss the oracles that deliver sufficient accuracy for such $\eps$, for the theory in  \cite{sass_arxiv} to apply. 

\textbf{Oracle Costs per Iteration.} 
In \cite{sass_arxiv}, it is shown that given the desired convergence tolerance $\eps$, sufficiently accurate oracles SS.0 and SS.1 can be implemented for any 
 step size parameter  $\alpha$ as follows: 
\begin{itemize}
	\item Zeroth-order oracle: 
%Furthermore, under the same conditions as in Proposition 5 of \cite{sass_arxiv}, 
Proposition 5 of \cite{sass_arxiv} shows that a sufficiently accurate zeroth-order oracle can be obtained by using a minibatch of size 
%$\frac{\sigma_f^2}{\epsef^2}$, where
\begin{equation}\label{eq:epsef}
\sc_0(\alpha) = 
\begin{cases}
{\cal O}(\sigma_f^2/\eps^4), &\text{for the  non-convex case,} \\
{\cal O}(\sigma_f^2/\eps^2), &\text{for the strongly convex case.}
\end{cases}
\end{equation}
Note that the cost of the zeroth-order oracle is independent of $\alpha$.
%\begin{equation}\label{eq:epsef}
%\epsef = 
%\begin{cases}
%C_{f1} \eps^2, &\text{if $\phi$ is non-convex,} \\
%C_{f2} \eps, &\text{if $\phi$ is $\beta$-strongly convex.}
%\end{cases}
%\end{equation}
\item First-order oracle: Proposition 6 of \cite{sass_arxiv} implies that a sufficiently accurate first-order oracle can be obtained by using a minibatch of size
%Now let us turn our attention to the  first-order oracle SS.1. It is shown in \cite{sass_arxiv} that  the stopping time $T_{\eps}$ generated by SASS algorithm is bounded as long as  
%	\begin{equation}\label{eq:epsg}
%\epsilon_g \leq  
%\begin{cases}
%	C_{1g} \eps, &\text{if $\phi$ is non-convex,} \\
%	C_{2g} \sqrt{\eps}, &\text{if $\phi$ is $\beta$-strongly convex.}
%\end{cases}
%\end{equation}
%	Here, $C_{1g} = \frac 16$ and $C_{2g} = \frac{\sqrt{2\beta}}{6}.$\footnote{This simplified version takes $\myeta = \frac16$ in Inequality 1 and 2 of \cite{sass_arxiv}. }
%	
%	
%	In addition, Proposition 6 of \cite{sass_arxiv} implies that  using a minibatch of size
\begin{equation}
\label{eq:sass_foc}
\sc_1(\alpha) = 
\begin{cases}
%O(\frac{M_c}{\eps^2} + \frac{M_v}{\alpha^2}), &\text{if $\phi$ is non-convex,} \\
%O(\frac{M_c}{\eps} + \frac{M_v}{ \alpha^2}), &\text{if $\phi$ is strongly convex.}
{\cal O}\left(\frac{M_c}{\eps^2} + \frac{M_v}{\min\{{\tau, \kappa\alpha\}^2}}\right), &\text{for the non-convex case,} \\
{\cal O}\left(\frac{M_c}{\eps} + \frac{M_v}{\min\{{\tau, \kappa\alpha\}^2}}\right), &\text{for the strongly convex case.}
\end{cases}
%\frac{M_c}{\delta\epsilon_g^2} + \frac{M_v{(1+\min\{\tau, \kappa\alpha\})}^2}{\delta \min\{\tau, \kappa\alpha\}^2}.
\end{equation}
The cost of the first-order oracle is indeed non-increasing in $\alpha$, so \Cref{ass:monotone} is satisfied. 
For simplicity of the presentation and essentially without loss of generality, we will assume $\tau \geq  \kappa\bar \alpha$.  
%For simplicity of the presentation, we will assume $\tau > \kappa\bar{\alpha}$, so that we can use $\kappa\bar\alpha$ instead of $\min\{\tau, \kappa\bar\alpha\}$ in the remainder of the analysis; the analysis can be carried out similarly if one keeps the $\tau$. 
%\begin{equation}
%\label{eq:sass_foc}
%\frac{M_c}{\delta\epsilon_g^2} + \frac{M_v{(1+\kappa\alpha)}^2}{\delta \kappa^2\alpha^2}.
%\end{equation}
%guarantees SS.1. 
%
%Thus, by picking $\epsilon_g$ accordingly we obtain the following sample cost of the first-order oracle. 
\end{itemize}

Substituting these bounds into \Cref{thm:sc_expected}, we obtain the following expected total sample complexity.

\begin{theorem}[Expected Total Sample Complexity of SASS]
\label{thm:sass_exp}
%Considering only the sample complexity used by the first-order oracle, we have
For the SASS algorithm applied to expected risk minimization, for any iteration \rev{number} $n \in \Z^+$, and any $\gamma > (2q)^{\frac12}$, we have

\begin{itemize}
\item Non-convex case:
\begin{align*}
\E[\SC(\min\{T_\eps, n\})] \leq  \;
& {\cal O}\left(\frac{\sigma_f^2}{\eps^4} \cdot n + \frac{M_c}{ \eps^2}
\cdot n
+ M_v \cdot n
\left(
n^{\log_{\frac{p}{q}}\left(\frac{1}{\gamma^2}\right)} 
\log_{\frac{p}{q}}(n)
\right)\right).
\end{align*}
\item Strongly convex case:
\begin{align*}
\E[\SC(\min\{T_\eps, n\})] \leq  \;
& {\cal O}\left(\frac{\sigma_f^2}{\eps^2} \cdot n + \frac{M_c}{ \eps}
\cdot n
+ M_v\cdot n
\left(
n^{\log_{\frac{p}{q}}\left(\frac{1}{\gamma^2}\right)} 
\log_{\frac{p}{q}}(n)
\right)\right).
\end{align*}
\end{itemize}

Moreover, if $\gamma \geq \left(\frac{q}{p}\right)^{\frac{\log c}{2\log n}}$ for some constant $c > 1$ (so that $n^{\log_{\frac{p}{q}}(\frac{1}{\gamma^2})} \leq c$), the above simplifies to
\begin{itemize}
	\item Non-convex case:
	\begin{equation}
\label{eq:sass_toc_exp_beta_non}
\E[\SC(\min\{T_\epsilon, n\})] \leq {\cal O} \left(n \left(\frac{\sigma_f^2}{\eps^4} + \frac{M_c}{\eps^2} +  M_v \log(n)\right)\right).
\end{equation}
	\item Strongly convex case:
	\begin{equation}
\label{eq:sass_toc_exp_beta_strong}
\E[\SC(\min\{T_\epsilon, n\})] \leq {\cal O} \left(n \left(\frac{\sigma_f^2}{\eps^2} + \frac{M_c}{\eps} +  M_v \log(n)\right)\right).
\end{equation}
\end{itemize}

\end{theorem}
\begin{proof}

Since the cost of each call to the zeroth-order oracle \eqref{eq:epsef} is independent of $\alpha$, the total function value sample complexity over $n$ iterations is simply obtained by multiplying \eqref{eq:epsef} by $n$.

The  cost of the first-order oracle \eqref{eq:sass_foc} consists of two parts, $\sc_1(\alpha) = \sc_{1,1}(\alpha) + \sc_{1,2}(\alpha)$. The first part, $\sc_{1,1}(\alpha)$,  is ${\cal O}(\frac{M_c}{\eps^2})$. Since this is \emph{independent} of $\alpha$, the total contribution of this part to the total gradient sample complexity over $n$ iterations is $n \, \sc_{1,1}(\alpha)$, which is bounded by ${\cal O}(\frac{M_c}{\eps^2}\,n)$.

The second part of the cost of the first-order oracle is $\sc_{1,2}(\alpha) := {\cal O}(\frac{M_v}{\min\{\tau, \kappa\alpha\}^2})$. By Theorem \ref{thm:sc_expected}, the total expected cost over $n$ iterations from this part is bounded above by:
\begin{align*}
%\E[\SC(\min\{T_\epsilon, n\}] &\leq
 &n \,\sum_{l=1}^{n} \min \left\{ 1, n\left(\frac{q}{p}\right)^l + \frac{2 \sqrt{pq}}{(1-2 \sqrt{p q})^{2}}(2q)^l\right \}\cdot \sc_{1,2}(\bar{\alpha} \gamma^{l})
	+ n \,\sc_{1,2}(\bar{\alpha}) \\
	&\leq n\underbrace{\sum_{l=1}^{n} \min\left\{1, n\left(\frac{q}{p}\right)^l\right\}\cdot \sc_{1,2}(\bar{\alpha} \gamma^{l})}_{=:A}
	 + \frac{2n\sqrt{pq}}{(1-2\sqrt{pq})^2} \underbrace{\sum_{l=1}^{n}\left(2q\right)^{l} \cdot \sc_{1,2}(\bar{\alpha} \gamma^{l})}_{=: B}
	+ n \sc_{1,2}(\bar{\alpha}).
\end{align*}

%We now plug in $\sc(\alpha) = O(\frac{M_v}{\min\{\tau, \alpha\}^2})$ to calculate upper bounds for $A$ and $B$. 
Note that the expression above only involves $\sc_{1,2}(\alpha)$ for $\alpha \leq \bar{\alpha}$. Therefore, by our earlier assumption that $\tau \geq \kappa\bar\alpha$, we have $\sc_{1,2}(\alpha) = {\cal O}(\frac{M_v}{\kappa^2\alpha^2}) = {\cal O}(\frac{M_v}{\alpha^2})$ (since $\kappa$ is a constant). We now use this to calculate upper bounds for $A$ and $B$. Using similar arguments as the proof of \Cref{thm:storm}, we have 
$$A = {\cal O}\left(\frac{M_v}{\bar{\alpha}^2} \, \log_{\frac{p}{q}}(n) \, n^{\log_{\frac{p}{q}}\left(\frac{1}{\gamma^2}\right)}  \right) 
 ~\text{and}~ B = {\cal O}\left(\frac{M_v}{\bar{\alpha}^2}  \right) .$$

Together with the previous arguments, the result follows.
\end{proof}

In \cite{sass_arxiv}, the iteration bound on $T_\epsilon$ is shown to be $\frac{C_1}{\eps^2}+\log_{1/\gamma} \frac{\alpha_0}{\bar{\alpha}}$ in the non-convex case, and $C_2\log \frac{1}{\eps}+\log_{1/\gamma} \frac{\alpha_0}{\bar{\alpha}}$ in the strongly convex case (for some $C_1$ and $C_2$ sufficiently large) in high probability. Thus, we can select $n$ appropriately and derive the high probability bound on the total sample complexity of SASS using \Cref{cor:sc_highprob}.

\begin{theorem}[High Probability Total Sample Complexity of SASS] 
\label{thm:sass_hp}
For the SASS algorithm applied to expected risk minimization, let $n \geq \frac{C_1}{\eps^2}+\log_{1/\gamma} \frac{\alpha_0}{\bar{\alpha}}$ in the non-convex case, and $n \geq C_2\log \frac{1}{\eps}+\log_{1/\gamma} \frac{\alpha_0}{\bar{\alpha}}$ in the strongly convex case.

%$n = \Theta(\frac{1}{\eps^2})$ if $\phi$ is non-convex, and $n = \Theta(\log \frac{1}{\eps})$ if $\phi$ is strongly convex.
For any $\omega  > 0$,
%For the SASS algorithm applied to expected risk minimization, with the same assumptions as in \cite{sass_arxiv}, for any $n\geq N$, $\omega  > 0$,
	with probability at least $1-2\exp\left(-C_3n\right) - {\cal O}(n^{-\omega })$ (for some $C_3 >0$), we have
	\begin{itemize}
		\item Non-convex case:
		\begin{equation}
	\label{eq:sass_toc_hp_non}
	\SC(T_{\eps})  \leq {\cal O}\left(\left(\frac{1}{\eps^2}+\log_{1/\gamma} \frac{\alpha_0}{\bar{\alpha}}\right)\cdot\left( \frac{\sigma_f^2}{\eps^4}  + \frac{M_c}{\eps^2} + \frac{M_v}{\eps^{4(1+\omega)\log_{2q}(\gamma)}} \right)\right).
	\end{equation}
		\item Strongly convex case:
		\begin{equation}
	\label{eq:sass_toc_hp_strong}
	\SC(T_{\eps})  \leq {\cal O}\left(\left ( \log\frac{1}{\eps}+\log_{1/\gamma} \frac{\alpha_0}{\bar{\alpha}}\right )\cdot \left(\frac{\sigma_f^2}{\eps^2}  + \frac{M_c}{\eps} + {M_v}\left(\log \frac{1}{\eps}\right)^{2(1+\omega)\log_{2q}(\gamma)}\right) \right).
	\end{equation}
	\end{itemize}
%	$1-\exp\left(-\frac{(p-\hat{p})^2}{2p}n\right) - \exp\left\{-\min\left\{\frac{s^2n}{8\nu^2},\frac{sn}{4b}\right\}\right\}- n^{-\omega } - cn^{-(1+\omega )}$\footnote{The $\nu$ and $b$ are the sub-exponential parameters of the zeroth-order oracle as defined in (1.1) of \cite{sass_arxiv}. That paper uses the moment-generating function definition of a sub-exponential random variable, which is an equivalent alternative to the definition in SS.0.
%%	an equivalent definition of the sub-exponential noise. 
%	},
%	$1-\exp\left(-\frac{(p-\hat{p})^2}{2p}n\right) - e^{-\min\left\{\frac{s^2n}{8\nu^2},\frac{sn}{4b}\right\}}- n^{-\omega } - cn^{-\frac{1}{2}(1+\omega )}$, 
%	\begin{equation}
%	\label{eq:sass_toc_hp_1}
%%	&\text{\textbf{Zeroth-order oracle:}} \;\; \SC(T_{\eps}-1)  \leq \frac{n\hat{\epsilon}^2}{\epsilon_f^2} \\
%%	\text{\textbf{First-order oracle:}} \;\; 
%	\SC(T_{\eps})  \leq n\cdot\left(  \frac{\sigma_f^2}{\epsef^2}  + \frac{M_c}{\delta\epsilon_g^2} + \frac{M_v(1+\kappa\bar{\alpha})^2}{\delta\kappa^2\bar{\alpha}^2\gamma^{2(1+\omega)\log_{1/2q}(n) + 2}} \right).
%	\end{equation}
%	where $\hat{p}$ is as defined in Theorem 3.8 of \cite{sass_arxiv}, and $N = \Theta(\frac{1}{\epsilon^2})$ if $\phi$ is non-convex, and $N = \Theta(\log \frac{1}{\epsilon})$ if $\phi$ is strongly convex.
	
	If $\gamma \in \left[\, \max\left\{\frac12, \left(\frac{1}{2q}\right)^{\frac{\log(2\beta)}{(1+\omega)\log n} } \right\}, \left(\frac{\bar{\alpha}}{\alpha_0}\right)^{\frac{1}{cn}}\,\right], $ where $\beta$ is any constant smaller than $\frac12$, and $c$ is any constant in $(0,1)$, the above simplifies to
	\begin{itemize}
		\item Non-convex case:
	\begin{equation}
	\label{eq:sass_toc_hp_beta_non}
	\SC(T_{\eps})  \leq {\cal O}\left(\frac{1}{\eps^2}\cdot\left( \frac{\sigma_f^2}{\eps^4} +  \frac{M_c}{\eps^2} + \frac{M_v}{\beta^2}\right)\right).
	\end{equation}
		\item Strongly convex case:
		\begin{equation}
	\label{eq:sass_toc_hp_beta_strong}
	\SC(T_{\eps})  \leq {\cal O}\left(\log\frac{1}{\eps}\cdot\left( \frac{\sigma_f^2}{\eps^2} +  \frac{M_c}{\eps} + \frac{M_v}{\beta^2}\right)\right).
	\end{equation}
	\end{itemize}

\end{theorem}
\begin{proof}
The bounds \eqref{eq:sass_toc_hp_non} and \eqref{eq:sass_toc_hp_strong} follow by using \eqref{eq:epsef} and \eqref{eq:sass_foc} in \Cref{cor:sc_highprob}, and Theorem 3.8 of \cite{sass_arxiv}.

The bounds \eqref{eq:sass_toc_hp_beta_non} and \eqref{eq:sass_toc_hp_beta_strong} follow from \eqref{eq:sass_toc_hp_non} and \eqref{eq:sass_toc_hp_strong}, respectively, by using the fact that $\gamma$ lies in the appropriate range and $\log_{1/\gamma} \frac{\alpha_0}{\bar{\alpha}} \leq cn$ with $c \in (0,1)$.

In the non-convex case, we have $\frac{1}{\eps^2} + \log_{1/\gamma} \frac{\alpha_0}{\bar{\alpha}} = {\cal O}(\frac{1}{\eps^2})$ and $\frac{M_v}{\eps^{4(1+\omega)\log_{2q}(\gamma)}} = \mathcal{O}(\frac{M_v}{\beta^2})$ when $\gamma$ lies in the specified range. It is worth noting that $c \in (0,1)$ implies there is some $n={\cal O}(\frac{1}{\eps^2})$ that satisfies $n \geq \frac{C_1}{\eps^2}+\log_{1/\gamma} \frac{\alpha_0}{\bar{\alpha}}$. The result for the strongly convex case follows similarly.

%If $\gamma = \left(\frac{1}{2q}\right)^{\frac{\log(2\beta)}{(1+\omega)\log n} }$, it follows that  
%	$$\log_{1/\gamma} \frac{\alpha_0}{\bar{\alpha}} = {\cal O}\left (\frac{(1+\omega)\log \frac{\alpha_0}{\bar{\alpha}}}{\log(2q) \log(2\beta)} \log n\right ), $$
%	which can then be ignored compared to ${\cal O}(n)$. \ml{$\gamma$ too close to 1 can still dominate n.}
\end{proof}

\begin{remark}
	\begin{enumerate}
	\item 		
	%Compared to the expected total sample complexity bound in \Cref{thm:sass_exp}, the high probability bound in \Cref{thm:sass_hp} is smaller by a  logarithmic factor in the term involving $M_v$. 
	We consider some of the implications of \Cref{thm:sass_hp} below. Similar implications hold for the expected total sample complexity.	
	\begin{itemize}
		\item In the non-convex case, from \eqref{eq:sass_toc_hp_beta_non}, the total function value sample complexity is ${\cal O}\left(\frac{\sigma_f^2}{\eps^6} \right)$  and the total gradient sample complexity is  ${\cal O}\left(\frac{M_c}{\eps^4} +  \frac{M_v}{\eps^2}\right)$. 
				In particular, the total gradient sample complexity matches that of SGD, and it essentially matches the complexity lower bound as described in \cite{arjevani2019lower} (for a different definition of complexity). Specifically, if $\sigma_f=0$  (i.e., function values are exact),
	the lower bound in \cite{arjevani2019lower} applies and the total sample complexity of SASS	matches it.

%   $O( \frac{M_v} {\eps^2}\log(\frac{1}{\eps}))$ for the gradient estimates.
 \item
If $M_c=0$ (sometimes referred to as the interpolation case), then the total gradient sample complexity reduces to ${\cal O}\left(\frac{M_v}{\eps^2} \right)$. Hence, the total gradient sample complexity matches that of SGD under interpolation \cite{bottou2018optimization,bach2011nonasymptotic}.

		\item In the strongly convex case, the total function value sample complexity is ${\cal O}( \frac{\sigma_f^2}{\eps^2} \log\frac{1}{\eps})$  and the total gradient sample complexity is ${\cal O}( (\frac{M_c}{\eps} + M_v)\log\frac{1}{\eps})$.
		%	Here, the big-$O$ is hiding a lot of constants.
		In particular, the total gradient sample complexity matches that of SGD up to a logarithmic term.
		\end{itemize}
	\rev{\item Our framework can be applied to the convex setting as well. We focus on the non-convex and strongly convex settings in this paper for brevity and to cleanly illustrate how our framework can be applied, since the convex case requires some additional complications in the presentation.  These complications are present in the previous papers that analyze iteration complexity bounds, and are not specific to this paper. For example, the stopping time in the convex setting for SASS is defined in terms of two parameters $\eps = (\eps_0, \eps_1)$ as follows: $T_\eps = \min\{k: \phi(X_k) - \phi(x^*) \leq \eps_0 ~\text{or}~ \norm{\nabla\phi(X_k)} \leq \eps_1\}$.}
	\end{enumerate}

\end{remark}

\section{Conclusion}
\label{sec:conclusion}
%Stochastic optimization algorithms, such as stochastic gradient descent and its variants, have gained widespread popularity in machine learning and signal processing. However, these algorithms have theoretical and practical limitations, including the high cost of tuning step sizes for each application. Adaptive stochastic optimization algorithms provide a promising solution, drawing from decades of deterministic optimization research and possessing strong convergence and worst-case complexity guarantees in diverse settings. However, analyzing these algorithms is challenging, and much of the previous work \cite{CS17,paquette2018stochastic,gratton2018complexity,blanchet2019convergence,berahas2019global,sass_arxiv,tr_witherrors} has focused on bounding iteration complexity rather than sample complexity.

We analyzed the behavior of the step size parameter in \Cref{alg:generic_stochastic}, an adaptive stochastic optimization framework that encompasses a wide class of algorithms analyzed in recent literature. We derived a high probability lower bound for this parameter, and as a result, developed a simple strategy for controlling this lower bound. 

For many settings, having a fixed lower bound on the step size parameter implies an upper bound on the cost of the oracles that compute the gradient and function estimates.  We developed a framework to analyze the expected and high probability total oracle complexity for this general class of algorithms, and illustrated the use of it by deriving total sample complexity bounds for two specific algorithms - the first-order stochastic trust-region (STORM) algorithm \cite{blanchet2019convergence} and a stochastic step search (SASS) algorithm \cite{sass_arxiv} in the expected risk minimization setting.   We showed that the sample complexity of both these algorithms essentially matches the complexity lower bound of first-order algorithms for stochastic non-convex optimization \cite{arjevani2019lower}, which was not known before. 

\section*{Acknowledgments}
	
	The authors wish to thank James Allen Fill for pointing out useful references and the proof of Proposition \ref{prop:Fill}.
	This work was partially supported by NSF Grants  TRIPODS 17-40796, CCF 2008434, CCF 2140057 and 
	ONR Grant N00014-22-1-2154. 
	% Miaolan Xie was partially supported by  a PhD Fellowship provided by  MunchRe.
	 Billy Jin was partially supported by NSERC fellowship PGSD3-532673-2019. 

\bibliographystyle{alpha} 
\bibliography{references}

\newcommand{\etalchar}[1]{$^{#1}$}
\begin{thebibliography}{TMDQ16}

\bibitem[ACD{\etalchar{+}}19]{arjevani2019lower}
Yossi Arjevani, Yair Carmon, John~C Duchi, Dylan~J Foster, Nathan Srebro, and
  Blake Woodworth.
\newblock Lower bounds for non-convex stochastic optimization.
\newblock {\em arXiv preprint arXiv:1912.02365}, 2019.

\bibitem[Ans60]{Anscombe1960}
Francis~John Anscombe.
\newblock Rejection of outliers.
\newblock {\em Technometrics}, 2:123--146, 1960.

\bibitem[BCCS21]{berahas2019theoretical}
Albert~S Berahas, Liyuan Cao, Krzysztof Choromanski, and Katya Scheinberg.
\newblock A theoretical and empirical comparison of gradient approximations in
  derivative-free optimization.
\newblock {\em Foundations of Computational Mathematics}, 2021.

\bibitem[BCMS19]{blanchet2019convergence}
Jose Blanchet, Coralia Cartis, Matt Menickelly, and Katya Scheinberg.
\newblock Convergence rate analysis of a stochastic trust-region method via
  supermartingales.
\newblock {\em INFORMS journal on optimization}, 1(2):92--119, 2019.

\bibitem[BCN18]{bottou2018optimization}
L{\'e}on Bottou, Frank~E Curtis, and Jorge Nocedal.
\newblock Optimization methods for large-scale machine learning.
\newblock {\em Siam Review}, 60(2):223--311, 2018.

\bibitem[BCS19]{berahas2019global}
Albert~S Berahas, Liyuan Cao, and Katya Scheinberg.
\newblock Global convergence rate analysis of a generic line search algorithm
  with noise.
\newblock {\em SIAM Journal on Optimization}, 2019.

\bibitem[BM11]{bach2011nonasymptotic}
F.~Bach and E.~Moulines.
\newblock Non-asymptotic analysis of stochastic approximation algorithms for
  machine learning.
\newblock In {\em Advances in Neural Information Processing Systems 24: 25th
  Annual Conference on Neural Information Processing Systems 2011. Proceedings
  of a meeting held 12-14 December 2011, Granada, Spain.}, pages 451--459,
  2011.

\bibitem[BXZ23]{sqp}
Albert~S. Berahas, Miaolan Xie, and Baoyu Zhou.
\newblock A sequential quadratic programming method with high probability
  complexity bounds for nonlinear equality constrained stochastic optimization,
  2023.

\bibitem[CBS22]{tr_witherrors}
Liyuan Cao, Albert~S Berahas, and Katya Scheinberg.
\newblock First-and second-order high probability complexity bounds for
  trust-region methods with noisy oracles.
\newblock {\em arXiv preprint arXiv:2205.03667}, 2022.

\bibitem[CS17]{CS17}
C.~Cartis and K.~Scheinberg.
\newblock Global convergence rate analysis of unconstrained optimization
  methods based on probabilistic models.
\newblock {\em Mathematical Programming}, 169(2):337–375, 2017.

\bibitem[Fel68]{feller_book}
William Feller.
\newblock {\em An Introduction to Probability Theory and Its Applications},
  volume~1.
\newblock Wiley, January 1968.

\bibitem[GRVZ18]{gratton2018complexity}
Serge Gratton, Cl{\'e}ment~W Royer, Lu{\'\i}s~N Vicente, and Zaikun Zhang.
\newblock Complexity and global rates of trust-region methods based on
  probabilistic models.
\newblock {\em IMA Journal of Numerical Analysis}, 38(3):1579--1597, 2018.

\bibitem[JSX21a]{sass_arxiv}
Billy Jin, Katya Scheinberg, and Miaolan Xie.
\newblock High probability complexity bounds for adaptive step search based on
  stochastic oracles.
\newblock {\em arXiv preprint arXiv:2106.06454}, 2021.

\bibitem[JSX21b]{aloe_neurips}
Billy Jin, Katya Scheinberg, and Miaolan Xie.
\newblock High probability complexity bounds for line search based on
  stochastic oracles.
\newblock In M.~Ranzato, A.~Beygelzimer, Y.~Dauphin, P.S. Liang, and J.~Wortman
  Vaughan, editors, {\em Advances in Neural Information Processing Systems},
  volume~34, pages 9193--9203, 2021.

\bibitem[KB14]{adam}
Diederik~P Kingma and Jimmy Ba.
\newblock Adam: A method for stochastic optimization.
\newblock {\em arXiv preprint arXiv:1412.6980}, 2014.

\bibitem[KPH15]{kim2015guide}
Sujin Kim, Raghu Pasupathy, and Shane~G Henderson.
\newblock A guide to sample average approximation.
\newblock {\em Handbook of simulation optimization}, pages 207--243, 2015.

\bibitem[MWX23]{qNewton}
Matt Menickelly, Stefan~M. Wild, and Miaolan Xie.
\newblock A stochastic quasi-newton method in the absence of common random
  numbers, 2023.

\bibitem[PS20]{paquette2018stochastic}
Courtney Paquette and Katya Scheinberg.
\newblock A stochastic line search method with expected complexity analysis.
\newblock {\em SIAM Journal on Optimization}, 30(1):349--376, 2020.

\bibitem[RJM17]{regier2017fast}
Jeffrey Regier, Michael~I Jordan, and Jon McAuliffe.
\newblock Fast black-box variational inference through stochastic trust-region
  optimization.
\newblock {\em Advances in Neural Information Processing Systems}, 30, 2017.

\bibitem[RVZ23]{rinaldi2023stochastic}
F~Rinaldi, LN~Vicente, and D~Zeffiro.
\newblock Stochastic trust-region and direct-search methods: A weak tail bound
  condition and reduced sample sizing.
\newblock 2023.

\bibitem[SHP18]{shashaani2018astro}
Sara Shashaani, Fatemeh~S Hashemi, and Raghu Pasupathy.
\newblock Astro-df: A class of adaptive sampling trust-region algorithms for
  derivative-free stochastic optimization.
\newblock {\em SIAM Journal on Optimization}, 28(4):3145--3176, 2018.

\bibitem[Spa99]{spsa}
James~C. Spall.
\newblock Stochastic optimization and the simultaneous perturbation method.
\newblock In {\em Proceedings of the 31st Conference on Winter Simulation:
  Simulation---a Bridge to the Future - Volume 1}, WSC '99, page 101–109, New
  York, NY, USA, 1999. Association for Computing Machinery.

\bibitem[SX23]{sarc}
Katya Scheinberg and Miaolan Xie.
\newblock Stochastic adaptive regularization method with cubics: A high
  probability complexity bound.
\newblock In {\em 2023 Winter Simulation Conference (WSC)}, To appear, 2023.

\bibitem[TMDQ16]{tan2016barzilai}
Conghui Tan, Shiqian Ma, Yu-Hong Dai, and Yuqiu Qian.
\newblock Barzilai-borwein step size for stochastic gradient descent.
\newblock {\em Advances in neural information processing systems}, 29, 2016.

\bibitem[YCKB18]{yin2018byzantine}
Dong Yin, Yudong Chen, Ramchandran Kannan, and Peter Bartlett.
\newblock Byzantine-robust distributed learning: Towards optimal statistical
  rates.
\newblock In {\em International Conference on Machine Learning}, pages
  5650--5659. PMLR, 2018.

\end{thebibliography}

\end{document}